\documentclass[11pt]{amsart}

\theoremstyle{plain}
\newtheorem{theorem}                {Theorem}      [section]
\newtheorem*{theorem*}                {Theorem \ref{thm:appl}}
\newtheorem{proposition}  [theorem]  {Proposition}
\newtheorem{corollary}    [theorem]  {Corollary}
\newtheorem{lemma}        [theorem]  {Lemma}

\theoremstyle{definition}

\newtheorem{remark}       [theorem]  {Remark}

\numberwithin{equation}{section}

\DeclareMathOperator{\trace}{trace}
\DeclareMathOperator{\grad}{grad}
\DeclareMathOperator{\Span}{span}
\DeclareMathOperator{\Imag}{Im}
\DeclareMathOperator{\Div}{div}

\begin{document}

\title[Biconservative surfaces in the $4$-dimensional Euclidean sphere]
{Biconservative surfaces in the $4$-dimensional Euclidean sphere}

\author{Simona Nistor, Cezar Oniciuc, Nurett\.{I}n Cenk Turgay, R\"uya Ye\u g\.{I}n \c{S}en}


\address{Faculty of Mathematics, Al. I. Cuza University of Iasi,
Blvd. Carol I, 11 \\ 700506 Iasi, Romania} \email{nistor.simona@ymail.com}

\address{Faculty of Mathematics, Al. I. Cuza University of Iasi,
Blvd. Carol I, 11 \\ 700506 Iasi, Romania} \email{oniciucc@uaic.ro}

\address{\.Istanbul Technical University, Faculty of Science and Letters, Department of Mathematics, 34469 Maslak, \.Istanbul, Türkiye}
\email{turgayn@itu.edu.tr}

\address{\.Istanbul Medeniyet University, Faculty of Engineering and Natural Sciences, Department of Mathematics, 34700 \"Usk\"udar, \.Istanbul, Türkiye} \email{ruya.yegin@medeniyet.edu.tr}

\thanks{The first author was supported by a grant of the Romanian Ministry of Research and Innovation, CNCS – UEFISCDI, project number PN-III-P1-1.1-PD-2019-0429, within PNCDI III. The third and fourth named authors were supported by a 3501 project of the Scientific and Technological Research Council of T\"urkiye (T\"UB\.ITAK) (Project Number:121F253).}

\begin{abstract}
In this paper, we study biconservative surfaces with parallel normalized mean curvature vector field ($PNMC$) in the $4$-dimensional unit Euclidean sphere $\mathbb{S}^4$. First, we study the existence and uniqueness of such surfaces. We obtain that there exists a $2$-parameter family of non-isometric abstract surfaces that admit a (unique) $PNMC$ biconservative immersion in $\mathbb{S}^4$. Then, we obtain the local parametrization of these surfaces in the $5$-dimensional Euclidean space $\mathbb{E}^5$.	
\end{abstract}



\maketitle

\section{Introduction}

A map $\varphi:\left(M^m,g\right)\to\left(N^n,h\right)$ between two Riemannian manifolds is said to be \textit{biharmonic} if it is a critical point of the \textit{bienergy} functional which is defined by  
$$
E_2:C^\infty\left(M^m,N^n\right)\to\mathbb R, \qquad E_2(\varphi)=\int_M |\tau (\varphi )|^2\ v_g,
$$
where $v_g$ denotes the volume element of $g$,
$$
\tau(\varphi ):=\trace\nabla \varphi_\ast
$$
is the tension field of $\varphi$ and $\varphi_\ast$ indicates the differential of $\varphi$. In \cite{Ji2}, the first and second variation of $E_2$ were obtained by Jiang, and he proved that $\varphi $ is biharmonic if and only if the associated Euler-Lagrange equation 
$$
\tau_2(\varphi):=-\Delta\tau(\varphi)-\trace R^N(\varphi_\ast,\tau(\varphi ))\varphi_\ast=0
$$
is satisfied, where $\tau_2$ is  the bitension field and $R^N$ denotes the curvature tensor field of $N^n$. 

On the other hand, in \cite{CMOP} Caddeo \textit{et al.} defined a \textit{biconservative} immersion as an isometric immersion satisfying the condition
$$
\langle\tau_2(\varphi ),\varphi_\ast \rangle=0
$$
(see Section \ref{preliminaries} for details). Obviously, biharmonic immersions are always biconservative. Thus, the biconservative immersions form a larger family of submanifolds including the biharmonic ones. In a space form, any hypersurface with constant mean curvature function ($CMC$) and any submanifold with parallel mean curvature vector field ($PMC$) are trivially biconservative. Therefore, in space forms we will exclude these cases.

In \cite{Hasanis-Vlachos}, Hasanis and Vlachos classified all biconservative hypersurfaces in $\mathbb{E}^4$. They called such hypersurfaces as $H$-hypersurfaces. After that, biconservative hypersurfaces and submanifolds in space forms have been studied by several geometers (see \cite{CMOP,DA2018,FNC2016,Yegin2022}, etc.). For example, in \cite{CMOP}, Caddeo \textit{et al.} classified, from an extrinsic local point of view, all non-$CMC$ biconservative surfaces in the three-dimensional space forms $N^3(c)$ and proved that for any $c$ there exists a one-parameter family of such surfaces (see also \cite{F15, Hasanis-Vlachos}). From a global point of view, this study was continued in \cite{MP2021, NO2020}. Furthermore, $PMC$ biconservative submanifolds in the product spaces $\mathbb{S}^n\times\mathbb{R}$ and $\mathbb{H}^n\times\mathbb{R}$ were studied in \cite{FNP2015,MFU2019}. 

The next step in the study of biconservative submanifolds in space forms is the study of non-$PMC$ biconservative surfaces in $N^4(c)$. This is also motivated by the paper \cite{MOR2016JGA}, where $CMC$ biconservative surface in $N^4(c)$ were studied, and by \cite{YeginTurgay2018}, where non-$CMC$ but $PNMC$ biconservative surfaces in $\mathbb{E}^4$ were classified. The study of $CMC$ biconservative surfaces in $N^4(c)$ proved to be rigid because these surfaces have to be either $PMC$, when $c\neq 0$, or certain cylinders, when $c=0$. On the other hand, the study of non-$CMC$ biconservative surfaces in $N^4(c)$ seems to be very complex and a full classification is hard to achieve in this case. Therefore, in \cite{YeginTurgay2018}, the authors added the natural hypothesis of being $PNMC$ and proved that the family of such surfaces is quite large.

In our paper, we continue the work in \cite{YeginTurgay2018} and investigate intrinsic and extrinsic properties of (non-$CMC$) $PNMC$ biconservative surfaces in $\mathbb{S}^4$. In Section \ref{preliminaries}, we present basic definitions and results that we are going to use. In Section \ref{sec-intrinsic}, we study intrinsic properties of  $PNMC$ biconservative surfaces in $\mathbb{S}^4$. In Theorem \ref{thm:uniquenessOfImmersions}, we prove that if an abstract surface admits two $PNMC$ biconservative immersions in $\mathbb{S}^4$, then these immersions must be congruent. Furthermore, in Theorem \ref{thm:intrinsicCharacterization}, we obtain a necessary and sufficient intrinsic condition in order for an abstract surface to admit a (unique) $PNMC$ biconservative immersion in $\mathbb{S}^4$. Essentially, that condition says that the level curves of the Gaussian curvature of the surface are circles with a certain (constant) curvature. The existence of such abstract surfaces is ensured by Theorem \ref{thm-existence} and Proposition \ref{other-parameters}. In Theorem \ref{thm-existence}, starting with a solution of a certain $ODE$, we construct locally the metric which satisfies the above condition. Alternatively, using appropriate coordinates with geometrical meaning, a more explicit form of the metric is given in Proposition \ref{other-parameters}. From relation \eqref{appropriate-coordinates}, we can see that the metric depends of two parameters. In Section \ref{sec-extrinsic}, we study extrinsic properties of $PNMC$ biconservative surfaces in $\mathbb{S}^4$. In particular, we obtain a local parametrization of our surfaces (see Theorem \ref{theorem-parametrization}).

\textbf{Conventions and notations.}

In this paper, all Riemannian metrics are indicated, in general, by the same symbol $\langle\cdot,\cdot\rangle$. Sometimes, when there is no confusion, we will omit to indicate the metric. We assume that the manifolds are connected and oriented, and we use the following sign conventions for the rough Laplacian acting on sections of the pull-back bundle $\varphi^{-1}\left(TN^n\right)$ and for the curvature tensor field, respectively:
$$
\Delta^{\varphi}=-\trace\left(\nabla^{\varphi}\nabla^{\varphi}-\nabla^{\varphi}_{\nabla}\right)
$$
and
$$
R(X,Y)Z=[\nabla_X,\nabla_Y]Z-\nabla_{[X,Y]}Z,
$$
where $\nabla$ is the Levi-Civita connection on $M$. 

In order to avoid trivial cases for our study of non-$CMC$ biconservative surfaces $M^2$ in $\mathbb{S}^4$, we will assume that the mean curvature function of the surface is positive, its gradient is different to zero at any point, and $M^2$ is completely contained in $\mathbb{S}^4$, i.e., for any open subset of $M^2$ there exists no great hypersphere $\mathbb{S}^3$ of $\mathbb{S}^4$ such that it lies in $\mathbb{S}^3$.

\section{Preliminaries}\label{preliminaries}
Let $\varphi:\left(M^m,g\right)\to \left(N^n,h\right)$ be an isometric immersion or, simply, let $M^{m}$ be a submanifold in $N^{n}$. We have the standard decomposition of $\varphi^{-1}\left(TN^n\right)$ into the direct sum of the tangent bundle $TM^m\equiv \varphi_\ast \left(TM^m\right)$ of $M^m$ and the normal bundle 
$$
NM^m=\displaystyle{\bigcup_{p\in M} \varphi_\ast \left(T_{p}M^m\right)^\perp}
$$ 
of $M^m$ in $N^n$. 

Locally, we can identify $M^m$ with its image by $\varphi$, $X$ with $\varphi_\ast (X)$ and $\nabla^\varphi_X \varphi_\ast(Y)$ with $\nabla^N_X Y$, where $\nabla^N$ is the Levi-Civita connection on $N^n$. Now, we recall the Gauss and the Weingarten formulas
\begin{equation*}
	\nabla^N_X Y=\nabla_X Y+B(X,Y),
\end{equation*}
and
\begin{equation*}
	\nabla^N_X \eta=-A_\eta(X)+\nabla^\perp_X\eta,
\end{equation*}
where $B\in C\left(\odot^2 T^\ast M^m\otimes NM^m\right)$ is called the second fundamental form of $M^m$ in $N^n$, $A_\eta\in C\left(T^\ast M^m\otimes TM^m\right)$ is the shape operator of $M^m$ in $N^n$ in the direction $\eta$, and $\nabla^\perp$ is the induced connection in the normal bundle.

In the particular cases when $N^n=\mathbb{S}^4$ or $N^n=\mathbb{E}^5$, we will denote the corresponding Levi-Civita connections by $\tilde{\nabla}$ or $\hat{\nabla}$, respectively.

The mean curvature vector field of $M^m$ in $N^n$ is defined by $H=(\trace B)/m\in C\left(NM^m\right)$, where the $\trace$ is considered with respect to the metric $g$. The mean curvature function is defined by $f=|H|$.

In this paper, we will assume that $H\neq 0$ at any point, so $f$ is a smooth positive function on $M^m$. We will denote $E_{m+1}=H/f$ and $A_{m+1}=A_{E_{m+1}}$. Also, a local orthonormal frame field in the normal bundle $NM^m$ of $M^m$ in $N^n$ will be indicated by $\left\{E_{m+1},\ldots, E_n\right\}$. 
 
We recall now the fundamental equations of submanifolds, i.e.,  the \textit{Gauss, Codazzi and Ricci equations}, that we will use later in this paper:

\begin{equation}\label{Gauss-equation}
\langle R^N(X,Y)Z,W\rangle=\langle R(X,Y)Z,W\rangle-\langle B(X,W),B(Y,Z)\rangle+\langle B(Y,W),B(X,Z)\rangle,
\end{equation}

\begin{equation}\label{Codazzi-equation}
	\left(\nabla_X A_\eta\right)(Y)-\left(\nabla_Y A_\eta\right)(X)=A_{\nabla_X^\perp \eta}(Y)-A_{\nabla_Y^\perp \eta}(X)-\left(R^N(X,Y)\eta\right)^\top,
\end{equation}
and 
\begin{equation}\label{Ricci-equation}
	\left(R^N(X,Y)\eta\right)^\perp=R^\perp(X,Y)\eta+B\left(A_\eta(X),Y\right)-B\left(X,A_\eta(Y)\right),
\end{equation} 
where $X, Y, Z, W\in C\left(TM^m\right)$ and $\eta\in C\left(NM^m\right)$.

We also recall the following characterization formulas of biharmonic submanifolds in Euclidean spheres.  

\begin{theorem}[\cite{B-Y. Chen, O2002}] Let $M^m$ be a submanifold in $\mathbb{S}^n$. Then $M^m$ is biharmonic if and only if
\begin{equation}\label{biharmonic-system}
\left\{
\begin{array}{l}
\Delta^\perp H+\trace B\left(\cdot,A_H (\cdot)\right)-mH=0 \\
4\trace A_{\nabla^\perp_{\cdot}H}\left(\cdot\right)+m\grad |H|^2=0
\end{array}
\right. .
\end{equation}
\end{theorem}
For more information concerning biharmonic submanifolds see, for example, \cite{FO2022, Ou-Chen}.

Now, if we fix a map $\varphi$ and let the domain metric vary, we obtain a new functional on the set $\mathcal{G}$ of all Riemannian metrics on $M^{m}$ defined by
$$
\mathcal{F}_{2}:\mathcal{G}\to \mathbb{R}, \quad \mathcal{F}_{2}(g)=E_{2}(\varphi).
$$
Critical points of this functional are characterized by the vanishing of the stress-energy tensor of the bienergy (see \cite{LMO}). This tensor field, denoted by $S_{2}$, was introduced in \cite{Jiang87} 
as
\begin{eqnarray*}
	S_{2}(X,Y)&=&\frac{1}{2}\vert \tau (\varphi)\vert ^{2}\langle X,Y \rangle +\langle \varphi_\ast, \nabla \tau (\varphi) \rangle \langle X, Y \rangle-\langle \varphi_\ast (X),\nabla_{Y} \tau (\varphi)\rangle
	\\
	&\ & -\langle \varphi_\ast (Y),\nabla_{X} \tau (\varphi)\rangle,
\end{eqnarray*}
and it satisfies
$$
\Div S_{2}=\langle \tau_{2}(\varphi), \varphi_\ast\rangle.
$$
We note that, for isometric immersions, $(\Div S_{2})^{\sharp} =-\tau_{2}(\varphi)^{\top}$, where $\tau_{2}(\varphi)^{\top}$ is the tangent part of the bitension field.

A submanifold $\varphi:M^{m} \to N^{n}$ of a Riemannian manifold $N^{n}$ is called \textit{biconservative} if $\Div S_{2}=0$.

\section{An intrinsic approach}\label{sec-intrinsic}

We will study non-$CMC$ biconservative surfaces with parallel normalized vector field in Euclidean sphere $\mathbb{S}^4$, i.e., $PNMC$ biconservative surfaces in $\mathbb{S}^4$.

We recall that (see \cite{FLO2021}) $M^2$ is a $PNMC$ biconservative surface in $\mathbb{S}^4$ if and only if
\begin{equation}\label{eq:PNMC-biconservative}
A_{3}(\grad f)=-f\grad f.
\end{equation}

Let us consider
$$
E_1=\frac{\grad f}{|\grad f|} \qquad \text{and} \qquad E_3=\frac{H}{f}.
$$

Because of orientation, we can consider positively oriented global orthonormal frame fields $\left\{E_1,E_2\right\}$ in the tangent bundle $TM^2$ and $\left\{E_3,E_4\right\}$ in the normal bundle $NM^2$.

Clearly, $E_2f=0$. Denoting $A_3=A_{E_3}$ and $A_4=A_{E_4}$, we have the following direct properties of our surfaces, which follow from \eqref{Gauss-equation}, \eqref{Codazzi-equation} and \eqref{Ricci-equation}.

\begin{theorem}\label{thm:fundamentalProperties}
Let $\varphi:\left(M^2,g\right)\to\mathbb{S}^4$ be a $PNMC$ biconservative immersion. Then, the following hold:
\begin{itemize}
\item [(i)] the Levi-Civita connection $\nabla$ of $M^2$ and the normal connection $\nabla^\perp$ of $M^2$ in $\mathbb{S}^4$ are given by
\begin{equation}\label{Levi-Civita-connection-f}
\nabla_{E_1}E_1=\nabla_{E_1}E_2=0, \quad \nabla_{E_2}E_1=-\frac{3}{4}\frac{E_1 f}{f}E_2, \quad \nabla_{E_2}E_2=\frac{3}{4}\frac{E_1 f}{f}E_1
\end{equation}
and
\begin{equation*}\label{normal-connection-f}
\nabla^\perp E_3=0, \qquad \nabla^\perp E_4=0;
\end{equation*}

\item [(ii)] the shape operators corresponding to $E_3$ and $E_4$ are given, with respect to $\left\{E_1, E_2\right\}$, by the matrices
\begin{equation*}\label{shape-operators-A3-A4-f}
A_3=\left(
\begin{array}{cc}
-f & 0 \\
 0 & 3f
\end{array}
\right),\quad
A_4=\left(
\begin{array}{cc}
cf^{3/2} & 0 \\
0 & -cf^{3/2}
\end{array}
\right),
\end{equation*}
where $c$ is a non-zero real constant;

\item [(iii)] the Gaussian curvature $K$ and the mean curvature function $f$ are related by
\begin{equation}\label{relation-K-f}
K=1-3f^2-c^2f^3,
\end{equation}
thus $1-K>0$ on $M^2$;

\item [(iv)] the mean curvature function $f$ satisfies
\begin{equation}\label{second-order-invariant-f}
f\Delta f+\left|\grad f\right|^2+\frac{4}{3}f^2-4f^4-\frac{4}{3}c^2f^5=0;
\end{equation}

\item [(v)] around any point of $M^2$ there exists a positively oriented local chart $X^f=X^f(u,v)$ such that
$$
\left(f\circ X^f\right)(u,v)=f(u,v)=f(u)
$$
and $f$ satisfies the following second order $ODE$
\begin{equation}\label{second-order-chart-f}
 f''f-\frac{7}{4}\left(f'\right)^2-\frac{4}{3}f^2+4f^4+\frac{4}{3}c^2f^5=0
\end{equation}
and the condition $f'>0$. The first integral of the above second order ODE is
\begin{equation}\label{first-integral-f}
\left(f'\right)^2-2C^2f^{7/2}+\frac{16}{9}f^2+16f^4+\frac{16}{9}c^2f^5=0,
\end{equation}
where $C$ is a non-zero real constant.
\end{itemize}
\end{theorem}

\begin{proof}
First, using \eqref{eq:PNMC-biconservative}, we have $A_3\left(E_1\right)=-fE_1$ and then, since $\trace A_3=2f$, we get $A_3\left(E_2\right)=3fE_2$.

As $\nabla^\perp E_3=0$, we have
$$\nabla^\perp E_4=0, \quad R^\perp(X,Y)E_3=0 \quad \text{and} \quad R^\perp(X,Y)E_4=0,
$$
for any $X,Y\in C\left(TM^2\right)$.

From the Ricci equation we obtain $B\left(E_1,E_2\right)=0$, so $\langle A_4\left(E_1\right), E_2\rangle=0$. On the other hand, since $\trace A_4=0$, we obtain that the matrix of $A_4$ with respect to $\left\{E_1, E_2\right\}$ is given by
\begin{equation*}
A_4=\left(
\begin{array}{cc}
\lambda & 0 \\
0 & -\lambda
\end{array}
\right),
\end{equation*}
for some smooth function $\lambda$ on $M^2$.

If we assume that $\lambda=0$ on $M^2$, or on an arbitrary open subset of $M^2$, we obtain that $\hat{\nabla}_X E_4=0$ for any $X\in C(T\mathbb{S}^4)$, hence $E_4$ is a constant vector field. Then, for any point $p\in M^2$, we can identify $(i\circ \varphi)(p)$ with its position vector in $\mathbb{E}^5$, where $i:\mathbb{S}^4\to \mathbb{E}^5$ is the canonical inclusion, and we have $(i\circ \varphi)(p)\perp E_4$. It follows that $(i\circ \varphi)(M)$ belongs to the hyperplane that passes through the origin and whose normal is $E_4$. Thus, $\varphi(M)\subset\mathbb{S}^3\subset\mathbb{S}^4$ and we get a contradiction. 

Therefore, $\lambda\neq 0$ at any point of an open dense subset of $M^2$. For simplicity, we will assume that $\lambda\neq 0$ at any point of $M^2$.

In order to obtain a more explicit expression of $\lambda$, we will use the Codazzi equation. More precisely, if we consider \eqref{Codazzi-equation} applied for $A_3$ we obtain the connection forms
\begin{equation}\label{connectionForms}
\omega^1_2\left(E_1\right)=0, \qquad \omega^1_2\left(E_2\right)= \frac{3}{4}\frac{E_1f}{f}.
\end{equation}
Then, from \eqref{Codazzi-equation} applied for $\eta=E_4$, we get
$$
E_1\lambda=\frac{3\lambda}{2}\frac{E_1 f}{f}, \qquad E_2\lambda=0.
$$
So,
$$
\frac{E_1\lambda}{\lambda}=\frac{3}{2}\frac{E_1f}{f},
$$
which is equivalent to
$$
E_1\left(\ln |\lambda|-\frac{3}{2}\ln f\right)=0.
$$
Moreover, as $E_2\lambda=0$ and $E_2f=0$, it follows that the function $\ln |\lambda|-3(\ln f)/2$ is constant, and therefore
$$
\lambda=cf^{3/2},
$$
where $c$ is a non-zero real constant.

We note that even if we had worked on a connected subset of the set of all points where $\lambda\neq 0$,  we would have obtain that the constant $c$ does not depend on that connected component.

Further, using \eqref{connectionForms}, it is easy to see that the Levi-Civita connection of $M^2$ is given by \eqref{Levi-Civita-connection-f}.

From the expressions of $A_3$ and $A_4$, we have
$$
B\left(E_1,E_1\right)=-fE_3+cf^{3/2}E_4, \qquad B\left(E_2,E_2\right)=3fE_3-cf^{3/2}E_4,
$$
and, thus, applying the Gauss equation, we obtain the relation between $K$ and $f$,
$$
K=1-3f^2-c^2f^3.
$$
Therefore, we proved $(i)$--$(iii)$ of the theorem. Next, we need to determine $(iv)$. In order to do this, we recall that
$$
d\omega^1_2\left(E_1,E_2\right)=K\left(\omega^1\wedge\omega^2\right)\left(E_1,E_2\right),
$$
i.e,
$$
E_1\left(\omega^1_2\left(E_2\right)\right)-E_2\left(\omega^1_2\left(E_1\right)\right)-\omega^1_2\left(\left[E_1,E_2\right]\right)=K.
$$
Using \eqref{Levi-Civita-connection-f} and \eqref{connectionForms}, the above relation can be rewritten as
$$
K=\frac{3}{16f^2}\left[4E_1\left(E_1 f\right)f-7\left(E_1f\right)^2\right].
$$
Now, replacing $K$ from \eqref{relation-K-f}, we get
\begin{equation}\label{secondDerivativefInvariant}
E_1\left(E_1f\right) f=\frac{7}{4}\left(E_1f\right)^2+\frac{4}{3}f^2-4f^4-\frac{4}{3}c^2f^5.
\end{equation}
From the expression of the Laplacian of $f$, we obtain
\begin{equation*}
\Delta f =-E_1\left(E_1f\right)+\frac{3}{4}\frac{\left(E_1f\right)^2}{f},
\end{equation*}
and therefore, \eqref{secondDerivativefInvariant} is equivalent to \eqref{second-order-invariant-f}. Consequently, we get $(iv)$.

Further, we will see that around any point of $M^2$ there exists a positively oriented local chart $X^f=X^f(u,v)$ such that
$$
\left(f\circ X^f\right)(u,v)=f(u,v)=f(u)
$$
and \eqref{second-order-invariant-f} is equivalent to a second order $ODE$.

Indeed, let $p_0\in M^2$ be an arbitrarily fixed point of $M^2$ and let $\gamma=\gamma(u)$ be an integral curve of $E_1$ with $\gamma(0)=p_0$. Let $\left\{\phi_v\right\}_{v\in\mathbb{R}}$ be the flow of $E_2$. We define the following local chart
$$
X^f(u,v)=\phi_v(\gamma(u))=\phi_{\gamma(u)}(v).
$$
We have 
\begin{align*}
&X^f(u,0)=\gamma(u),\\
&X^f_u(u,0)=\gamma'(u)=E_1(\gamma(u))=E_1(u,0), \\
&X^f_v(u,v)=\phi'_{\gamma(u)}(v)=E_2\left(\phi_{\gamma(u)}(v)\right)=E_2(u,v)
\end{align*}
for any $(u,v)$. Clearly, $\left\{X^f_u,X^f_v\right\}$ is positively oriented. 

If we write the Riemannian metric $g$ on $M^2$ in local coordinates as
$$
g=\langle \cdot, \cdot \rangle=g_{11}\ du^2+2g_{12}\ du\  dv+g_{22}\  dv^2,
$$
we get $g_{22}=\left|X^f_v\right|^2=\left|E_2\right|^2=1$, $g_{11}(u,0)=1$ and $g_{12}(u,0)=0$. Then, it is easy to see that
$$
E_1=\frac{1}{\sigma}\left(X^f_u-g_{12}\ X^f_v\right)=\sigma\ \grad u,
$$
where $\sigma=\sigma(u,v)=\sqrt{g_{11}-g_{12}^2}>0$ and $\sigma(u,0)=1$.

Let
$$
\left(f\circ X^f\right)(u,v)=f(u,v).
$$
Since $E_2f=0$, i.e., $X^f_v\ f=0$, it follows that
$$
f(u,v)=f(u,0)=f(u), \qquad \forall (u,v).
$$
As
\begin{align*}
\left[E_1,E_2\right]f &=\left(\nabla_{E_1}E_2-\nabla_{E_2}{E_1}\right)f=0\\
&=E_1\left(E_2f\right)-E_2\left(E_1f\right),
\end{align*}
it follows that $E_2\left(E_1 f\right)=0$ and, therefore,
\begin{equation}
\label{relatiesigma}
X^f_v\left(\frac{f'}{\sigma}\right)=0.
\end{equation}
Clearly,
$$
f'=E_1f=|\grad f|>0.
$$
Then, as $X^f_v\ f'=0$, from \eqref{relatiesigma} we obtain $X^f_v(\sigma)=0$, i.e.,
$$
\sigma(u,v)=\sigma(u,0)=\sigma(u)=1, \qquad \forall (u,v)
$$
and
\begin{equation}\label{E1}
E_1=X^f_u-g_{12}X^f_v=\grad u.
\end{equation}
Since $f=f(u)$, from \eqref{secondDerivativefInvariant} and \eqref{E1}, we obtain that $f$ satisfies equation \eqref{second-order-chart-f}.

Our next objective is to find the first integral of \eqref{second-order-chart-f}. First, we note that
\begin{equation}\label{eq:helpfullIntegralPrime}
\frac{1}{2}\left(\frac{\left(f'\right)^2}{f^{7/2}}\right)'=\frac{f'f''}{f^{7/2}}-\frac{7}{4}\frac{\left(f'\right)^3}{f^{9/2}}.
\end{equation}
Then, we multiply \eqref{second-order-chart-f} by $f'/f^{9/2}$ and using \eqref{eq:helpfullIntegralPrime}, we have
$$
\left(\frac{1}{2}\frac{\left(f'\right)^2}{f^{7/2}}+\frac{8}{9}c^2f^{3/2}+8\sqrt{f}+\frac{8}{9f^{3/2}}\right)'=0.
$$
Integrating the above relation and multiplying the result by $2f^{7/2}$, one obtains \eqref{first-integral-f}. Hence, we get $(v)$.
\end{proof}

The next result shows that, given an abstract surface $\left(M^2,g\right)$, if it admits a $PNMC$ biconservative immersion $\varphi$ in $\mathbb{S}^4$, then it is unique. In particular, up to the sign, the constant $c$ is unique, depends on $\left(M^2,g\right)$ and it is not an indexing constant. This phenomenon is different from the case of minimal surfaces in a $3$-dimensional space form where, given an abstract surface $\left(M^2,g\right)$ satisfying a certain intrinsic condition, there exists a one-parametric family of minimal immersions (see \cite{L70,MM2015,R95}).

\begin{theorem}\label{thm:uniquenessOfImmersions}
If an abstract surface $\left(M^2,g\right)$ admits two $PNMC$ biconservative immersions in $\mathbb{S}^4$, then these immersions differ by an isometry of $\mathbb{S}^4$.
\end{theorem}

\begin{proof}
Let $\varphi_1:M^2\to\mathbb{S}^4$ and $\varphi_2:M^2\to\mathbb{S}^4$ be two $PNMC$ biconservative immersions. From Theorem \ref{thm:fundamentalProperties} we know that
\begin{equation}\label{Kc1c2}
K=1-3f_1^2-c_1^2f_1^3=1-3f_2^2-c_2^2f_2^3, \qquad c_1,c_2\neq 0.
\end{equation}
It follows that
\begin{equation}\label{gradKc1c2}
XK=\left(-6f_1-3c_1^2f_1^2\right)Xf_1=\left(-6f_2-3c_2^2f_2^2\right)Xf_2, \qquad X\in C\left(TM^2\right).
\end{equation}
Let us denote $^iE_1=\grad f_i/\left|\grad f_i\right|$ and consider $^iE_2$ to be the tangent vector field such that $\left\{^iE_1,^iE_2\right\}$ is a positively oriented global orthonormal frame field in the tangent bundle, for $i=1,2$. From \eqref{gradKc1c2} we get
$$
\grad K=\left(-6f_i-3c_i^2f_i^2\right)\grad f_i, \qquad i=1,2,
$$
and since $\grad f_i\neq 0$ at any point, we have
$$
\frac{\grad K}{\left|\grad K\right|}=-\frac{\grad f_i}{\left|\grad f_i\right|}, \qquad i=1,2.
$$
Therefore, $^1E_1=\ ^2E_1=E_1$ and $^1E_2=\ ^2E_2=E_2$. We note that $E_1$ and $E_2$ do not depend on the immersion, but only on the surface.

Using the expression of $\nabla_{E_2}E_1$ in \eqref{Levi-Civita-connection-f}, we have
$$
\frac{E_1f_1}{f_1}=\frac{E_1f_2}{f_2}.
$$
By combining the above equation and \eqref{Kc1c2} we get $f_1=f_2$ and $c_1=\varepsilon c_2$, where $\varepsilon=\pm1$.

Therefore, $^1A_3=\ ^2A_3$ and $^1A_4=\varepsilon\ ^2A_4$. Now, let us define the vector bundle isomorphism $\psi:N_{\varphi_1}M^2\to N_{\varphi_2}M^2$ by
$$
\psi\left(^1H\right)=\ ^2H,\quad \text{i.e.,} \quad \psi\left(^1E_3\right)=\ ^2E_3, \quad  \text{and} \quad \psi\left(^1E_4\right)=\varepsilon\ ^2E_4.
$$
Then, since $f_1=f_2$, we have
\begin{equation*}
\begin{array}{lll}
\langle\psi\left(^1E_\alpha\right), \psi\left(^1E_\beta\right)\rangle=\langle ^1E_\alpha,^1E_\beta\rangle=\delta_{\alpha\beta}, \qquad \alpha,\beta=3,4, \\\\
\psi\left(^1\nabla_{E_i}^\perp\left(^1E_\alpha\right)\right)=\ ^2\nabla_{E_i}^\perp\left(^2E_\alpha\right)=0, \qquad i=1,2,\ \alpha=3,4,\\\\
\psi\left(^1B\left(E_i,E_j\right)\right)=\ ^2B\left(E_i,E_j\right), \qquad i,j=1,2.
\end{array}
\end{equation*}
Thus, from the Fundamental Theorem of Submanifolds (see, for example, \cite{DT2019}), there exists an isometry $\Phi$ of $\mathbb{S}^4$ such that
$$
\Phi\circ \varphi_1=\varphi_2 \qquad \text{and} \qquad \Phi_{\ast|N_1M}=\psi.
$$
\end{proof}

\begin{remark}
We know that, in fact, $E_1$ and $E_2$ belong to the intrinsic geometry of $\left(M^2,g\right)$. Moreover, from \eqref{Kc1c2} and \eqref{gradKc1c2}, we obtain
$$
\nabla_{E_2}{E_1}=-\frac{3}{4}\frac{E_1f}{f}E_2=-\frac{E_1K}{4\left(K-1+f^2\right)}E_2.
$$
Thus, $f$ and then $c^2$, are uniquely determined by $\left(M^2,g\right)$.
\end{remark}

We give now another intrinsic property of the domain $\left(M^2,g\right)$ of a $PNMC$ biconservative immersion in $\mathbb{S}^4$, which will play a fundamental role in the existence of our isometric immersions.

\begin{proposition}\label{Prop:intrinsicPropertyDomain}
Let $\varphi:\left(M^2,g\right)\to \mathbb{S}^4$ be a $PNMC$ biconservative immersion. Then, $\grad K \neq 0$ everywhere and the level curves of $K$ are circles of $M^2$ with positive constant signed curvature
\begin{equation}\label{curvatureOfLevelCurves}
\kappa=\frac{3}{4}\frac{E_1f}{f}=-\frac{\left|\grad K\right|}{4\left(K-1+f^2\right)}.
\end{equation}
Moreover, in the positively oriented local chart  $X^f=X^f(u,v)$,
$$
\kappa(u)=\frac{3}{4}\frac{f'(u)}{f(u)}
$$
and $\kappa$ satisfies the following $ODE$
\begin{equation}\label{eq:kappa1}
3\kappa\kappa''' -26\kappa^2\kappa''-3\kappa'\kappa''+72\kappa^3\kappa'-32\kappa^3-32\kappa^5=0.
\end{equation}
\end{proposition}

\begin{proof}
We have seen that, since $\grad f\neq 0$, also $\grad K\neq 0$ and
$$
\frac{\grad K}{\left|\grad K\right|}=-\frac{\grad f}{\left|\grad f\right|}.
$$
We define
$$
\tilde{E}_1=\frac{\grad K}{\left|\grad K\right|}=-E_1 \qquad \text{and} \qquad \tilde{E}_2=-E_2
$$
such that $\left\{\tilde{E}_1,\tilde{E}_2\right\}$ is a positively oriented global orthonormal frame field.

As $E_2f=0$, from \eqref{relation-K-f}, it follows that $E_2K=0$ and so $\tilde{E}_2K=0$, i.e., the integral curves of $\tilde{E}_2$ are the level curves of $K$.

We note that $\left\{\tilde{E}_2,-\tilde{E}_1\right\}$ is positively oriented and taking into account \eqref{Levi-Civita-connection-f} and \eqref{relation-K-f} we get
$$
\nabla_{\tilde{E}_1}\tilde{E}_2=0, \quad \nabla_{\tilde{E}_2}\tilde{E}_1=-\frac{|\grad K|}{4\left(K-1+f^2\right)}\tilde{E}_2, \quad \nabla_{\tilde{E}_2}\tilde{E}_2=\frac{|\grad K|}{4\left(K-1+f^2\right)}\tilde{E}_1.
$$
Thus,
\begin{align*}
\left[\tilde{E}_1,\tilde{E}_2\right]K &=0\\
&=-\tilde{E}_2\left(\tilde{E}_1 K\right)
\end{align*}
and so $\tilde{E}_2\left(\tilde{E}_1 K\right)=0$. But this implies
$$
\tilde{E}_2\left(-\frac{|\grad K|}{4\left(K-1+f^2\right)}\right)=0.
$$
Now, as the integral curves of $\tilde{E}_2$ are the level curves of $K$, it follows that
$$
\kappa=-\frac{|\grad K|}{4\left(K-1+f^2\right)}
$$
is constant along the level curves of $K$.

If we consider $\Gamma=\Gamma(v)$ an integral curve of $\tilde{E}_2$, since $\left|\tilde{E}_2\right|=1$, one gets that $\Gamma$ is a curve parametrized by arc-length.

As $\left\{\Gamma', -\tilde{E}_{1|\Gamma}\right\}$ is positively oriented and
$$
\nabla_{\Gamma'}\Gamma'=\kappa_{|\Gamma}\left(-\tilde{E}_{1|\Gamma}\right) \qquad \text{and} \qquad \nabla_{\Gamma'}\left(-\tilde{E}_{1}\right)=-\kappa_{|\Gamma}\Gamma',
$$
where $\kappa_{|\Gamma}$ is a constant, it follows that $\Gamma$ is a circle of $M^2$ with constant curvature $\kappa$.

Therefore, the level curves of $K$ are circles of $M^2$ with constant curvature
$$
\kappa=-\frac{|\grad K|}{4\left(K-1+f^2\right)}.
$$
Finally, if we consider the positively oriented local chart $X^f=X^f(u,v)$ as in Theorem \ref{thm:fundamentalProperties}, we immediately note that $\kappa(u)=3f'(u)/(4f(u))$ and thus
$$
\frac{f''}{f}=\frac{16}{9}\kappa^2 + \frac{4}{3}\kappa'.
$$
Replacing the above relation in \eqref{second-order-chart-f}, by a straightforward computation one obtains
\begin{equation}
\label{ralatia2}
c^2=\frac{1+\kappa^2-\kappa'}{f^3} - \frac{3}{f}.
\end{equation}
Next, taking the derivative of \eqref{ralatia2}, one gets
\begin{equation}
\label{relatia3}
0=-\kappa'' + 6\kappa\kappa' - 4\kappa - 4\kappa^3 + 4\kappa f^2,
\end{equation}
and taking one more derivative, we find
\begin{equation}
\label{relatia4}
0=-3\kappa''' + 18\left(\kappa'\right)^2 + 18\kappa\kappa'' - 12\kappa' - 36\kappa^2\kappa' + 4\left(3\kappa' + 8\kappa^2\right)f^2.
\end{equation}
Finally, from \eqref{relatia3} and \eqref{relatia4}, by standard computations one obtains that $\kappa$ satisfies \eqref{eq:kappa1}.
\end{proof}


\begin{lemma}\label{lemma:uniqueSolution}
Let $\left(M^2,g\right)$ be an abstract surface such that $1-K>0$ and $c$ be a non-zero real constant. Then, the polynomial equation
\begin{equation}\label{polynomialEquation}
1-K-3f^2-c^2f^3=0
\end{equation}
and the condition $f>0$ determine $f$ uniquely.
\end{lemma}

\begin{proof}
Let $p_0\in M^2$ be an arbitrarily fixed point and consider the polynomial equation
\begin{equation}\label{eqKx}
K\left(p_0\right)=1-3x^2-c^2x^3, \qquad x>0.
\end{equation}
Consider the function $h:(0,\infty)\to(-\infty,1)$ given by 
$$
h(x)=1-3x^2-c^2x^3.
$$
Clearly, $h$ is smooth and $h'(x)<0$ for any $x>0$. As
$$
\lim_{x\searrow 0}h(x)=1 \qquad \text{and} \qquad \lim_{x\nearrow\infty} h(x)=-\infty,
$$
we get that $h$ is a smooth diffeomorphism. Therefore, the solution of \eqref{eqKx} is $x=h^{-1}\left(K\left(p_0\right)\right)$, i.e., $f\left(p_0\right)=h^{-1}\left(K\left(p_0\right)\right)$. Thus, $f=h^{-1}\circ K$ is a smooth function being the composition of two smooth functions, is positive and is the unique solution of \eqref{polynomialEquation}.
\end{proof}

\begin{remark}
Further, we denote by $f_K=h^{-1}\circ K$ the positive smooth solution of \eqref{polynomialEquation}, so the relation \eqref{curvatureOfLevelCurves} can be rewritten as
$$
\kappa=-\frac{\left|\grad K\right|}{4\left(K-1+f_K^2\right)}.
$$
\end{remark}

\begin{proposition}\label{prop:EquivalentConditions}
Let $\left(M^2,g\right)$ be an abstract surface such that $1-K>0$ and $\grad K\neq 0$ everywhere. Consider some non-zero real constant $c$ and $f_K$ the positive solution of \eqref{polynomialEquation}. Let
$$
\tilde{E}_1=\frac{\grad K}{\left|\grad K\right|} \qquad \text{and} \qquad \tilde{E}_2\in C\left(TM^2\right)
$$
be two vector fields on $M^2$ such that $\left\{\tilde{E}_1,\tilde{E}_2\right\}$ is a positively oriented global orthonormal frame field in the tangent bundle $TM^2$. Then, the following conditions are equivalent:
\begin{itemize}
\item [(i)] the level curves of $K$ are circles of $M^2$ with positive constant signed curvature
\begin{equation}\label{curvatureOfLevelCurvesTilde}
\kappa=-\frac{\left|\grad K\right|}{4\left(K-1+f_K^2\right)}=-\frac{\tilde{E}_1 K}{4\left(K-1+f_K^2\right)};
\end{equation}

\item[(ii)] the Levi-Civita connection $\nabla$ of $M^2$ is given by
\begin{equation}\label{Levi-Civita-connection-K}
\nabla_{\tilde{E}_1}\tilde{E}_1=\nabla_{\tilde{E}_1}\tilde{E}_2=0, \quad \nabla_{\tilde{E}_2}\tilde{E}_1=-\frac{\tilde{E}_1 K}{4\left(K-1+f_K^2\right)}\tilde{E}_2, \quad \nabla_{\tilde{E}_2}\tilde{E}_2=\frac{\tilde{E}_1 K}{4\left(K-1+f_K^2\right)}\tilde{E}_1.
\end{equation}
\end{itemize}
\end{proposition}

\begin{proof}
First, we assume that $(i)$ holds and want to prove $(ii)$. As $\left\{\tilde{E}_2,-\tilde{E}_1\right\}$ is positively oriented and the level curves of $K$, i.e. the integral curves of $\tilde{E}_2$, are circles of $M^2$ with constant signed curvature $\kappa$ given in \eqref{curvatureOfLevelCurvesTilde}, it follows that $\tilde{E}_2\kappa=0$. Thus, we have
$$
\nabla_{\tilde{E}_2}\tilde{E}_2=\kappa\left(-\tilde{E}_1\right)=\frac{\tilde{E}_1 K}{4\left(K-1+f_K^2\right)}\tilde{E}_1
$$
and
$$
\nabla_{\tilde{E}_2}\left(-\tilde{E}_1\right)=-\kappa\tilde{E}_2=\frac{\tilde{E}_1 K}{4\left(K-1+f_K^2\right)}\tilde{E}_2.
$$
Now, using \eqref{curvatureOfLevelCurvesTilde} and $\tilde{E}_2f_K=0$, we obtain $\tilde{E}_2\left(\tilde{E}_1K\right)=0$.

We note that
\begin{align*}
\left[\tilde{E}_1,\tilde{E}_2\right]K & =\tilde{E}_1\left(\tilde{E}_2 K\right)-\tilde{E}_2\left(\tilde{E}_1 K\right)=0 \\
& = \left(\nabla_{\tilde{E}_1}\tilde{E}_2-\nabla_{\tilde{E}_2}\tilde{E}_1\right)(K)=\left(\nabla_{\tilde{E}_1}\tilde{E}_2\right)(K).
\end{align*}
Therefore, we have $\left(\nabla_{\tilde{E}_1}\tilde{E}_2\right)(K)=\langle \nabla_{\tilde{E}_1}\tilde{E}_2, \grad K \rangle=0$, and so $\langle  \nabla_{\tilde{E}_1}\tilde{E}_2, \tilde{E}_1 \rangle=0$. Consequently,
$$
\nabla_{\tilde{E}_1}\tilde{E}_2=0 \qquad \text{and} \qquad \nabla_{\tilde{E}_1}\tilde{E}_1=0.
$$
Finally, the converse is easy to prove since from the expression of the Levi-Civita connection we obtain $\tilde{E}_2\left(\tilde{E}_1K\right)=0$ and thus $\tilde{E}_2\kappa=0$. Then, we conclude from the Frenet formulas applied for the integral curves of $\tilde{E}_2$.

\end{proof}

The following result represents an intrinsic characterization for the existence of $PNMC$ biconservative immersions in $\mathbb{S}^4$.

\begin{theorem}\label{thm:intrinsicCharacterization}
Let $\left(M^2,g\right)$ be an abstract surface. Then $M^2$ admits locally a (unique) $PNMC$ biconservative embedding in $\mathbb{S}^4$ if and only if $1-K>0$, $\grad K\neq 0$ everywhere and the level curves of $K$ are circles of $M^2$ with positive constant signed curvature
$$
\kappa=-\frac{\left|\grad K\right|}{4\left(K-1+f_K^2\right)},
$$
where $f_K$ is the positive solution of \eqref{polynomialEquation} for some non-zero real constant $c$.
\end{theorem}

\begin{proof}
The direct implication has been already proved in Proposition \ref{Prop:intrinsicPropertyDomain}. It remains to show the converse one.

Let
$$
\tilde{E}_1=\frac{\grad K}{\left|\grad K\right|} \qquad \text{and} \qquad \tilde{E}_2\in C\left(TM^2\right)
$$
be two vector fields on $M^2$ such that $\left\{\tilde{E}_1,\tilde{E}_2\right\}$ is a positively oriented global orthonormal frame field. Clearly, $\tilde{E}_2K=0$ and $\tilde{E}_1 K=|\grad K|$. Now, since $f_K$ is the positive solution of \eqref{polynomialEquation}, for some non-zero real constant $c$, it follows that $\tilde{E}_2f_K=0$ and $\grad f_K\neq 0$ at any point. We set
$$
E_1=\frac{\grad f_K}{\left|\grad f_K\right|}=-\tilde{E}_1 \qquad \text{and} \qquad E_2=-\tilde{E}_2,
$$
and it is clear that $\left\{E_1,E_2\right\}$ is a positively oriented global orthonormal frame field.

Next, let $\Upsilon=M^2\times\mathbb{E}^2$ be the trivial vector bundle of rank two over $M^2$. We define $\sigma_3$ and $\sigma_4$  by
\begin{align*}
\sigma_3(p) &=(p,(1,0)), \qquad p\in M^2, \\
\sigma_4(p) &=(p,(0,1)), \qquad p\in M^2,
\end{align*}
which form the canonical global frame field of $\Upsilon$, by $a$ the metric on $\Upsilon$ defined by
$$
a\left(\sigma_\alpha,\sigma_\beta\right)=\langle \sigma_\alpha,\sigma_\beta \rangle=\delta_{\alpha\beta}, \qquad \alpha,\beta=3,4,
$$
and by $\nabla^a$ the connection on $\Upsilon$ given by
$$
\nabla^a_{E_i}\sigma_\alpha=0, \qquad i=1,2,\  \alpha=3,4.
$$
Clearly, the pair $\left(\nabla^a,a\right)$ is a Riemannian structure, i.e.,
$$
X\langle\sigma,\rho\rangle=\langle\nabla^a_X\sigma,\rho\rangle+\langle \sigma, \nabla^a_X\rho\rangle, \qquad X\in C\left(TM^2\right), \ \rho,\sigma\in C\left(\Upsilon\right),
$$
and the curvature tensor field
$$
R^a\left(E_i,E_j\right)\sigma_\alpha=0, \qquad i=1,2,\  \alpha=3,4.
$$
Let us define $B^a:C\left(TM^2\right)\times C\left(TM^2\right)\to C\left(\Upsilon\right)$ by
\begin{equation*}
\left\{
\begin{array}{lll}
B^a\left(E_1,E_1\right)=-f_K\sigma_3+cf_K^{3/2}\sigma_4 \\\\
B^a\left(E_1,E_2\right)=B^a\left(E_2,E_1\right)=0 \\\\
B^a\left(E_2,E_2\right)= 3f_K\sigma_3-cf_K^{3/2}\sigma_4
\end{array}
\right..
\end{equation*}
Consider $A^a_\alpha\in C\left(End\left(TM^2\right)\right)$, $\alpha=3,4$, given by
$$
\langle A^a_\alpha\left(E_i\right),E_j\rangle=\langle B^a\left(E_i,E_j\right),\sigma_\alpha\rangle, \qquad i,j=1,2, \ \alpha=3,4.
$$
It is easy to see that $A^a_\alpha$ satisfies, formally, the Gauss, Codazzi and Ricci equations for surfaces in $\mathbb{S}^4$. Therefore, according to the Fundamental Theorem of Submanifolds, locally, there exists an isometric embedding $\varphi:\left(M^2,g\right)\to \mathbb{S}^4$ and a vector bundle isometry $\psi:\Upsilon\to NM$ such that
$$
\nabla^\perp \psi=\psi\nabla^a \qquad \text{and} \qquad B=\psi\circ B^a,
$$
i.e.,
\begin{equation*}
\left\{
\begin{array}{lll}
\nabla^\perp_{E_i}\left(\psi\left(\sigma_\alpha\right)\right)=\psi\left(\nabla^a_{E_i}\sigma_\alpha\right)=0\\\\
B\left(E_i,E_j\right)=\psi\left(B^a\left(E_i,E_j\right)\right), \qquad i,j=1,2,\ \alpha=3,4.
\end{array}
\right.
\end{equation*}
As the last step, we denote by $E_\alpha=\psi\left(\sigma_\alpha\right)$ , $\alpha=3,4$. With all the above notations we can see that, with respect to $\left\{E_1, E_2\right\}$, the shape operators $A_{E_\alpha}=A_\alpha$ have the matrices
\begin{equation*}
A_3=\left(
\begin{array}{cc}
-f_K & 0 \\
0 & 3f_K
\end{array}
\right),\quad
A_4=\left(
\begin{array}{cc}
cf^{3/2}_K & 0 \\
0 & -cf^{3/2}_K
\end{array}
\right),
\end{equation*}
the mean curvature function of the immersion $\varphi$ is $f=f_K$ and $E_3=H/f$.

Now, we prove that $\varphi$ is a $PNMC$ biconservative immersion in $\mathbb{S}^4$ with $f>0$, $\grad f\neq 0$ at any point, and the surface is completely contained in $\mathbb{S}^4$.

First, by straightforward computations, we obtain that $\varphi$ is a $PNMC$ biconservative immersion in $\mathbb{S}^4$ with $f>0$ and $\grad f\neq 0$ at any point. It remains to prove that any open subset of $M^2$ cannot lie in some totally geodesic $\mathbb{S}^3\subset\mathbb{S}^4$. Indeed, since
$$
B\left(E_1,E_1\right)=-fE_3+cf^{3/2}E_4
$$
and
$$
B\left(E_2,E_2\right)= 3fE_3-cf^{3/2}E_4,
$$
it follows that $\left\{B\left(E_1,E_1\right)(p), B\left(E_2,E_2\right)(p)\right\}$ is a basis in $N_pM^2$, for any $p\in M^2$.

Let us consider the first normal bundle $N_1$ defined by
\begin{align*}
N_1=\Span \Imag(B) & = \Span\left\{B\left(E_1,E_1\right), B\left(E_2,E_2\right)\right\} \\
                & = \Span\left\{E_3,E_4\right\}.
\end{align*}
Clearly, $\dim N_1=2$.

Assume that $M^2$, or an arbitrary open subset of $M^2$, lies in a totally geodesic hypersurface $\mathbb{S}^3\subset\mathbb{S}^4$ defined by
\begin{equation*}
\mathbb{S}^3:\left\{
\begin{array}{ll}
\langle \overline{x},\eta_0\rangle=0 \\\\
\left|\overline{x}\right|=1
\end{array}
\right. .
\end{equation*}
The normal bundle of $M^2$ in $\mathbb{S}^4$ can be written as
$$
NM=\Span \left\{E_3,E_4\right\}=\Span\left\{\eta_1,\eta_0\right\},
$$
where $\eta_1$ is a unit normal vector field to $M^2$ in $\mathbb{S}^3$. Moreover, we denote by $\tilde{\nabla}$ the Levi-Civita connection on $\mathbb{S}^4$ and it is easy to see that, since
\begin{align*}
\hat{\nabla}_X\eta_0 & = 0 \\
& = \tilde{\nabla}_X \eta_0 - \langle X,\eta_0\rangle (i\circ\varphi) = \tilde{\nabla}_X \eta_0 = -A_{\eta_0}(X)+\nabla^\perp_X \eta_0,
\end{align*}
we have $A_{\eta_0}(X)=0$, for any $X\in C\left(TM^2\right)$. Therefore, 
$$
B(X,Y)=\langle A_{\eta_1}(X),Y\rangle\eta_1
$$
and so, $N_1=\Span\left\{\eta_1\right\}$, i.e., $\dim N_1=1$, which is a contradiction.
\end{proof}

Further, we will effectively establish the existence of $PNMC$ biconservative immersions in $\mathbb{S}^4$. As we have seen in Theorem \ref{thm:intrinsicCharacterization}, to prove the existence of a such immersion is equivalent to prove the existence of an abstract surface $\left(M^2,g\right)$ such that the level curves of $K$ are circles with a certain curvature $\kappa$. For this purpose, the idea is to work with the curvature $\kappa$, not with $K$ itself, and to show that $\kappa$ satisfies a third order $ODE$ (see \eqref{eq:kappa2}). First, we give

\begin{theorem}\label{thm:RelationBetweenKandk}
Let $\left(M^2,g\right)$ be an abstract surface such that $1-K>0$, $\grad K\neq 0$ everywhere and the level curves of $K$ are circles of $M^2$ with positive constant signed curvature
$$
\kappa=-\frac{\left|\grad K\right|}{4\left(K-1+f_K^2\right)},
$$
where $f_K$ is the positive solution of \eqref{polynomialEquation} for some non-zero real constant $c$. Then, around any point of $M^2$ there exists a positively oriented local chart $X^K=X^K(u,s)$ such that
\begin{equation*}
\begin{array}{ll}
\left(K\circ X^K\right)(u,s)=K(u,s)=K(u),\\\\
\left(\kappa\circ X^K\right)(u,s)=\kappa(u,s)=\kappa(u)
\end{array}
\end{equation*}
and
\begin{equation}\label{eq:K-kappa-kappa'}
K(u)=-\kappa^2(u)-\kappa'(u).
\end{equation}
Moreover, the metric $g$ has the form
$$
g(u,s)=du^2+\theta^2(u)ds^2,
$$
where
$$
\theta(u)=e^{\int_{u_0}^u\kappa(\tau)\  d\tau},
$$
and its Gaussian curvature is
$$
K(u)=-\frac{\theta''(u)}{\theta(u)}.
$$
\end{theorem}

\begin{proof}
First, recall that
$$
\tilde{E}_1=\frac{\grad K}{\left|\grad K\right|} \qquad \text{and} \qquad \tilde{E}_2\in C\left(TM^2\right)
$$
are two vector fields on $M^2$ such that $\left\{\tilde{E}_1,\tilde{E}_2\right\}$ is a positively oriented global orthonormal frame field.

For the first part of the theorem, let $p_0\in M^2$ be an arbitrarily fixed point of $M^2$ and $\gamma=\gamma(u)$ be an integral curve of $\tilde{E}_1=\grad K /|\grad K|$ with $\gamma(0)=p_0$. Let $\left\{\phi_v\right\}_{v\in\mathbb{R}}$ be the flow of $\tilde{E}_2$. We define the following positively oriented local chart
$$
X^K(u,v)=\phi_v(\gamma(u))=\phi_{\gamma(u)}(v)
$$
and we have
\begin{align*}
	&X^K(u,0)=\gamma(u), \\
	&X^K_u(u,0)=\gamma'(u)=\tilde{E}_1(\gamma(u))=\tilde{E}_1(u,0),  \\
	&X^K_v(u,v)=\phi'_{\gamma(u)}(v)=\tilde{E}_2\left(\phi_{\gamma(u)}(v)\right)=\tilde{E}_2(u,v).
\end{align*}
Clearly, $\tilde{E}_1 K=|\grad K|$ and $\tilde{E}_2K=0$. For any $u$, the curve $v\to \phi_{\gamma(u)}(v)$ is a level curve of $K$ and so
\begin{equation*}
\begin{array}{ll}
\left(K\circ X^K\right)(u,v)=K(u,v)=K(u),\\\\
\left(\kappa\circ X^K\right)(u,v)=\kappa(u,v)=\kappa(u).
\end{array}
\end{equation*}
Using \eqref{Levi-Civita-connection-K}, we get
$$
\kappa=-\langle \nabla_{\tilde{E}_2}\tilde{E}_2,\tilde{E}_1\rangle
$$
and thus
\begin{align*}
\kappa' & = \tilde{E}_1\kappa = -\tilde{E}_1\langle \nabla_{\tilde{E}_2}\tilde{E}_2,\tilde{E}_1\rangle \\ & = -\langle \nabla_{\tilde{E}_1}\nabla_{\tilde{E}_2}\tilde{E}_2,\tilde{E}_1\rangle = -\langle R(\tilde{E}_1,\tilde{E}_2)\tilde{E}_2,\tilde{E}_1\rangle - \langle\nabla_{[\tilde{E}_1,\tilde{E}_2]}\tilde{E}_2,\tilde{E_1}\rangle \\
& = -K-\kappa^2.
\end{align*}

For the second part of the theorem, let  $p_0\in M^2$ be an arbitrarily fixed point of $M^2$. In order to find a simpler expression of the metric $g$ on $M^2$, we look for new local coordinates $u$ and $s$, where $s=v/\theta(u)$, $\theta(u)>0$, such that $X^K_u$ and $X^K_s$ are orthogonal. The parameter $u$ remains the same and we only change homothetically the parameter $v$ of the level curves of $K$. Clearly, in the new coordinates, $K(u,s)=K(u)$ and $\kappa(u,s)=\kappa(u)$.

So, we consider
$$
X^K(u,s)= \phi_{\theta(u)s}(\gamma(u))=\phi_{\gamma(u)}(\theta(u)s).
$$
It is easy to see that $X^K(u,0)=\gamma(u)$ and
$$
X^K_u(u,0)=\tilde{E}_1(u,0), \qquad X^K_s(u,s)=\theta(u)\tilde{E}_2(u,s).
$$
Therefore, for any $(u,s)$ we have
$$
g_{11}(u,0)=1, \qquad g_{12}(u,0)=0, \qquad g_{22}(u,s)=\theta^2(u).
$$
By some direct computations, we deduce
$$
\tilde{E}_1(u,s)=\frac{1}{\sigma(u,s)}\left(X^K_u(u,s)-\frac{g_{12}(u,s)}{\theta^2(u)}X^K_s(u,s)\right),
$$
where
$$
\sigma(u,s)=\sqrt{g_{11}(u,s)-\frac{g_{12}^2(u,s)}{\theta^2(u)}}
$$
and $\sigma(u,0)=1$. So, $\tilde{E}_1K=K'(u)/\sigma(u,s)$.

Further, using the hypothesis that the level curves of $K$, which are the integral curves of $\tilde{E}_2$, are circles with constant signed curvature, we get
$$
\kappa=\kappa(u,s)=\kappa(u)=-\frac{K'(u)}{4\sigma(u,s)\left(K(u)-1+f_K^2(u)\right)}.
$$
Therefore,
$$
\sigma=\sigma(u,s)=\sigma(u)=\sigma(u,0)=1,
$$
which implies
$$
\tilde{E}_1(u,s)=X^K_u(u,s)-\frac{g_{12}(u,s)}{\theta^2(u)}X^K_s(u,s)
$$
and
\begin{equation}\label{g11}
g_{11}(u,s)-\frac{g_{12}^2(u,s)}{\theta^2(u)}=1.
\end{equation}
Using \eqref{Levi-Civita-connection-K}, we have $\nabla_{\tilde{E}_2}\tilde{E}_2=\kappa\left(-\tilde{E}_1\right)$, so $\kappa=-\langle \nabla_{\tilde{E}_2}\tilde{E}_2,\tilde{E}_1\rangle$. Next, as $\tilde{E}_1=X^K_u-g_{12}X^K_s/\theta^2$, $\tilde{E}_2=X^K_s/\theta$ and 
$$ 
\nabla_{X^K_s}X^K_s=\Gamma_{22}^1 X^K_u+\Gamma_{22}^2X^K_s,
$$
after some computations, we reach
$$
\kappa(u)=-\frac{\Gamma_{22}^1(u,s)}{\theta^2(u)}.
$$
Further, from the expression of the Christoffel symbol in terms of the metric coefficients, we obtain
$$
\Gamma_{22}^1(u,s)=\frac{\partial g_{12}}{\partial s}(u,s)-\theta(u)\theta'(u),
$$
and thus
$$
\kappa(u)=-\frac{1}{\theta^2(u)}\left(\frac{\partial g_{12}}{\partial s}(u,s)-\theta(u)\theta'(u)\right).
$$
Taking into account that $g_{12}(u,0)=0$, from the above relation it follows that
\begin{equation}\label{g12}
g_{12}(u,s)=\left(-\theta^2(u)\kappa(u)+\theta(u)\theta'(u)\right)s.
\end{equation}
Now, if we impose that $g_{12}(u,s)=0$, from \eqref{g11} we have $g_{11}(u,s)=1$, and from  \eqref{g12} we get
\begin{equation}\label{eq:kappa-theta'-theta}
\frac{\theta'(u)}{\theta(u)}=\kappa(u).
\end{equation}
Integrating the above relation, it follows that
$$
\theta(u)=e^{\int_{u_0}^u\kappa(\tau)\  d\tau}.
$$
Finally, we obtain that
$$
g(u,s)=du^2+\theta^2(u)ds^2.
$$
From \eqref{eq:K-kappa-kappa'} and \eqref{eq:kappa-theta'-theta} or, alternatively, from \eqref{Gauss-equation} and computing all the Christoffel symbols, we can see that
$$
K(u)=-\frac{\theta''(u)}{\theta(u)}.
$$
\end{proof}

\begin{remark}
Around an arbitrary point $p_0\in M^2$ we have
$$
X^K(u,v)=X^f(-u,-v),
$$
where $X^K=X^K(u,v)$ was defined in the first part of the proof of Theorem \ref{thm:RelationBetweenKandk}.
\end{remark}

\begin{theorem}\label{thm:RelationBetweenkandf_K}
Let $\left(M^2,g\right)$ be an abstract surface such that $1-K>0$, $\grad K\neq 0$ everywhere and the level curves of $K$ are circles of $M^2$ with positive constant signed curvature
\begin{equation}\label{kfK}
\kappa=-\frac{\left|\grad K\right|}{4\left(K-1+f_K^2\right)},
\end{equation}
where $f_K$ is the positive solution of \eqref{polynomialEquation} for some non-zero real constant $c$. Then, we have

\begin{equation}\label{eq:f_K}
f_K^2=\frac{\kappa''+6\kappa\kappa'+4\kappa+4\kappa^3}{4\kappa},
\end{equation}
\begin{equation}\label{eq:c^2}
c^2=-\frac{2\sqrt{\kappa}\left(3\kappa''+14\kappa\kappa'+8\kappa+8\kappa^3\right)}{\left(\kappa''+6\kappa\kappa'+4\kappa+4\kappa^3\right)^{3/2}}
\end{equation}
and the signed curvature $\kappa$ satisfies the following $ODE$
\begin{equation}\label{eq:kappa2}
3\kappa\kappa'''+26\kappa^2\kappa''-3\kappa'\kappa''+72\kappa^3\kappa'+32\kappa^3+32\kappa^5=0.
\end{equation}
\end{theorem}

\begin{proof}
Let $X^K=X^K(u,s)$ be the positively oriented local chart defined in Theorem \ref{thm:RelationBetweenKandk}. It is easy to see that from \eqref{kfK} we get
\begin{equation}\label{eq:fK1}
f_K^2=\frac{-K'-4(K-1)\kappa}{4\kappa},
\end{equation}
and, from \eqref{eq:K-kappa-kappa'} we have 
\begin{equation}\label{eq:K'}
K'=-2\kappa\kappa'-\kappa''.
\end{equation}
Now, by combining \eqref{eq:fK1} and \eqref{eq:K'}, we obtain \eqref{eq:f_K}. 

Further, since $f_K$ is the positive solution of \eqref{polynomialEquation}, we get 
\begin{equation}\label{eq:c^2-proof}
c^2=\frac{1-3f_K^2-K}{f_K^3}.
\end{equation} 
On the one hand, using the expressions of $f_K$ from \eqref{eq:f_K} and of $K$ from \eqref{eq:K-kappa-kappa'}, we can infer that \eqref{eq:c^2-proof} is equivalent to \eqref{eq:c^2}.

On the other hand, taking the derivatives of \eqref{eq:c^2-proof} and of \eqref{eq:f_K}, one gets
\begin{equation}\label{eq:firstderivative-c^2}
3f_Kf_K'\left(1-K-f_K^2\right)+K'f_K^2=0
\end{equation}
and
\begin{equation}\label{eq:firstderivative-f_K^2}
f_Kf_K'=\frac{\kappa\kappa'''-\kappa'\kappa''+6\kappa^2\kappa''+8\kappa^3\kappa'}{8\kappa^2},
\end{equation}
respectively.

Finally, using \eqref{eq:K-kappa-kappa'}, \eqref{eq:f_K}, \eqref{eq:K'}, and \eqref{eq:firstderivative-f_K^2} in \eqref{eq:firstderivative-c^2}, we obtain \eqref{eq:kappa2}.
\end{proof}

\begin{remark}
We note that, since $X^K(u,s)=X^f(-u,-s)$, if we change $u$ by $-u$, then from \eqref{eq:kappa2} we recover \eqref{eq:kappa1}. If $\kappa=\kappa(u)$ is a solution of \eqref{eq:kappa2}, then the set of points where $\kappa'(u)\neq 0$ is open and dense in the domain of $\kappa$. Moreover, as we have seen in the proof of Theorem \ref{thm:RelationBetweenKandk}, the function $\theta$ from the expression of the metric $g$ is determined up to a positive multiplicative constant. By an appropriate change of coordinates, taking into account that the solutions of \eqref{eq:kappa2} are invariant under translations of the argument, we can always assume that the constant is one. Thus, we can assume
$$
\theta(u)=e^{\int_0^u\kappa(\tau)\  d\tau}.
$$
\end{remark}

Now, we are ready to obtain in a constructive way the existence result for $PNMC$ biconservative immersions in $\mathbb{S}^4$. This is a kind of a converse of Theorem \ref{thm:RelationBetweenkandf_K}. Starting with a solution of \eqref{eq:kappa2}, we will define locally a metric that satisfies all the required properties in order to ensure the existence of a $PNMC$ biconservative immersion in $\mathbb{S}^4$. To get the solution of \eqref{eq:kappa2} we must consider certain appropriate initial conditions. 

\begin{theorem}\label{thm-existence}
Consider the third order ODE \eqref{eq:kappa2} and the initial conditions
\begin{equation}\label{initialConditions1}
\left\{
\begin{array}{lll}
\kappa(0)=\kappa_0 \\
\kappa'(0)=\kappa'_0 \\
\kappa''(0)=\kappa''_0
\end{array}
\right.
\end{equation} 	
such that they satisfy
\begin{equation}\label{intialConditions2}
\left\{
\begin{array}{lll}
\kappa_0>0 \\
\kappa'_0>-1-\kappa^2_0 \\
-4\kappa_0-6\kappa_0\kappa'_0-4\kappa^3_0<\kappa''_0<\frac{1}{3}\left(-8\kappa_0-14\kappa_0\kappa'_0-8\kappa^3_0\right).
\end{array}
\right.
\end{equation} 	
Let $\kappa=\kappa(u)$ be the solution of \eqref{eq:kappa2} and \eqref{initialConditions1} and assume that
\begin{equation}\label{intialConditions3}
\left\{
\begin{array}{lll}
\kappa(u)>0 \\
\kappa'(u)>-1-\kappa^2(u) \\
-4\kappa(u)-6\kappa(u)\kappa'(u)-4\kappa^3(u)<\kappa''(u)<\frac{1}{3}\left(-8\kappa(u)-14\kappa(u)\kappa'(u)-8\kappa^3(u)\right), \qquad \forall u.
\end{array}
\right.
\end{equation} 	
Define
\begin{equation}\label{definition-theta}
\theta(u)=e^{\int_0^u\kappa(\tau)\  d\tau}
\end{equation}
and
\begin{equation}\label{definition-g(u,s)}
g(u,s)=du^2+\theta^2(u)ds^2.
\end{equation}
Then, the metric $g$ satisfies
\begin{itemize}
\item [(i)] $K(u)=-\kappa^2(u)-\kappa'(u)$;

\item[(ii)] $1-K>0$;

\item [(iii)] $\grad K \neq 0$ everywhere;

\item [(iv)]
$$
c^2=-\frac{2\sqrt{\kappa}\left(3\kappa''+14\kappa\kappa'+8\kappa+8\kappa^3\right)}{\left(\kappa''+6\kappa\kappa'+4\kappa+4\kappa^3\right)^{3/2}}
$$
is a positive constant;

\item [(v)]
$$
f_K^2=\frac{\kappa''+6\kappa\kappa'+4\kappa+4\kappa^3}{4\kappa}
$$
is positive and satisfies
\begin{equation}\label{formula-K}
K=1-3f_K^2-c^2f_K^3,
\end{equation}
where $c^2$ is given at (iv);

\item[(vi)] the Levi-Civita connection of $M^2$ is given by
$$
\nabla_{\tilde{E}_1}\tilde{E}_1=\nabla_{\tilde{E}_1}\tilde{E}_2=0, \quad \nabla_{\tilde{E}_2}\tilde{E}_1=\kappa\tilde{E}_2, \quad \nabla_{\tilde{E}_2}\tilde{E}_2=-\kappa\tilde{E}_1,
$$
where
$$
\tilde{E}_1=\frac{\grad K}{\left|\grad K\right|} \qquad \text{and} \qquad \tilde{E}_2\in C\left(TM^2\right)
$$
form a positively oriented global orthonormal frame field on $M^2$.
\end{itemize}
\end{theorem}

\begin{proof}
In order to prove $(i)$--$(iii)$, we first use the expression \eqref{definition-g(u,s)} of the metric $g=g(u,s)$ and, by standard computation, we deduce that 
$$
K(u)=-\frac{\theta''(u)}{\theta(u)}.
$$
Then, using \eqref{definition-theta}, we get
\begin{equation}\label{eq:K-proof}
K(u)=-\kappa^2(u)-\kappa'(u).
\end{equation}  
Now, since $1-K=1+\kappa^2+\kappa'$, from the second inequality of \eqref{intialConditions3}, it is clear that $1-K>0$.

To prove that $\grad K \neq 0$ at any point, we note that 
$$
\grad K=K'X^K_u.
$$
Next, we can see that $K'$ is positive. Indeed, if we take the derivative of \eqref{eq:K-proof}, we obtain that $K'>0$ if and only if 
\begin{equation}\label{inequality-k''}
\kappa''<-2\kappa\kappa'.
\end{equation}
But, from the last inequality of \eqref{intialConditions3}, we know that
$$
\kappa''<\frac{1}{3}\left(-8\kappa-14\kappa\kappa'-8\kappa^3\right),
$$
thus, in order to prove \eqref{inequality-k''}, it is enough to show that 
$$
\frac{1}{3}\left(-8\kappa-14\kappa\kappa'-8\kappa^3\right)<-2\kappa\kappa'.
$$
Furthermore, the last inequality is equivalent to
$$
\kappa'>-1-\kappa^2,
$$
which is true, according to the second inequality of \eqref{intialConditions3}. Therefore $K'>0$.

In order to prove $(iv)$ and $(v)$, we set 
$$
c^2=-\frac{2\sqrt{\kappa}\left(3\kappa''+14\kappa\kappa'+8\kappa+8\kappa^3\right)}{\left(\kappa''+6\kappa\kappa'+4\kappa+4\kappa^3\right)^{3/2}}
$$
and
$$
f_K^2=\frac{\kappa''+6\kappa\kappa'+4\kappa+4\kappa^3}{4\kappa}.
$$
From the first and the last inequalities of \eqref{intialConditions3}, we see that the above two quantities are positive.

Since $K$ satisfies \eqref{eq:K-proof}, by some straightforward computations, we can rewrite $c^2$ as
\begin{equation}\label{formula-c^2}
c^2=\frac{1-3f_K^2-K}{f_K^3}
\end{equation}
and thus \eqref{formula-K} holds.

We still have to show that the quantity from the right hand side of \eqref{formula-c^2} is constant, i.e., its derivative vanishes. Indeed, we first note that
$$
\left(\frac{1-3f_K^2-K}{f_K^3}\right)'=0
$$
if and only if
\begin{equation}\label{formula-equivalent-derivative-zero}
3f_Kf_K'\left(1-K-f_K^2\right)+K'f_K^2=0.
\end{equation}
Since $\kappa\neq 0$, using the definition of $f_K^2$ and relation \eqref{eq:K-proof}, we have that \eqref{formula-equivalent-derivative-zero} is equivalent to
\begin{equation}\label{formula-equvialent-derivative-zero-2}
\left(2\kappa\kappa'+\kappa''\right)\left(3\kappa\kappa'''+26\kappa^2\kappa''-3\kappa'\kappa''+72\kappa^3\kappa'+32\kappa^3+32\kappa^5\right)=0.
\end{equation}
Furthermore, from \eqref{inequality-k''}, equation \eqref{formula-equvialent-derivative-zero-2} is equivalent to the vanishing of its second parenthesis. Finally, we note that 
$$
3\kappa\kappa'''+26\kappa^2\kappa''-3\kappa'\kappa''+72\kappa^3\kappa'+32\kappa^3+32\kappa^5=0
$$
is precisely equation \eqref{eq:kappa2} that $\kappa$ satisfies.

Therefore, the quantity denoted by $c^2$ is a positive constant.

Further, to prove the last item, we consider the global vector field $\tilde{E}_1=\grad K/\left|\grad K\right|$. Since $g(u,s)=du^2+\theta^2(u)ds^2$ and $\grad K=K'(u)X^K_u$, it follows that 
$$
\tilde{E}_1=X^K_u. 
$$
Let 
$$
\tilde{E}_2=\frac{1}{\theta(u)}X^K_s\in C\left(TM^2\right).
$$
Clearly, $\left\{\tilde{E}_1,\tilde{E}_2\right\}$ form a positively oriented global orthonormal frame field on $M^2$. 

Using the definition of $\theta$, it is easy to see that $$
\kappa(u)=\frac{\theta'(u)}{\theta(u)}
$$ 
and, expressing the Christoffel symbols in terms of the metric $g$, by standard computations we reach the Levi-Civita connection on $M^2$
$$
\nabla_{\tilde{E}_1}\tilde{E}_1=\nabla_{\tilde{E}_1}\tilde{E}_2=0, \quad \nabla_{\tilde{E}_2}\tilde{E}_1=\kappa\tilde{E}_2, \quad \nabla_{\tilde{E}_2}\tilde{E}_2=-\kappa\tilde{E}_1.
$$
\end{proof}

\begin{remark}
It is not difficult to find examples of triplets $\left(\kappa_0,\kappa_0',\kappa_0''\right)$ that satisfy \eqref{intialConditions2}.
\end{remark}

In the following, from a \textit{local} point of view (in our context of $PNMC$ biconservative surfaces in $\mathbb{S}^4$), we will show that all non-isometric abstract surfaces $\left(M^2,g\right)$ that admit a $PNMC$ biconservative immersion in $\mathbb{S}^4$ are determined by the solutions $\kappa=\kappa(u)$ of \eqref{eq:kappa2}.

Indeed, let $\left(M^2_1,g_1\right)$ and $\left(M^2_2,g_2\right)$ be two abstract surfaces given by
$$
g_1\left(u_1,s_1\right)=du^2_1+\theta^2_1\left(u_1\right)ds^2_1, \quad \theta_1(0)=1,
$$
where $\kappa_1=\kappa_1\left(u_1\right)$ is a solution of \eqref{eq:kappa2} and
$$
g_2\left(u_2,s_2\right)=du^2_2+\theta^2_2\left(u_2\right)ds^2_2, \quad \theta_2(0)=1,
$$
where $\kappa_2=\kappa_2\left(u_2\right)$ is also solution of \eqref{eq:kappa2}. We can give the next result.

\begin{proposition}
If there exists an isometry $\Psi:\left(M^2_1,g_1\right)\to\left(M^2_2,g_2\right)$, then
\begin{itemize}
\item[(i)] $\Psi\left(u_1,s_1\right)=\left(u_1+\alpha,\beta s_1+\gamma\right)$, where $\alpha$ and $\gamma$ are some real constants and $\beta=\theta_1(-\alpha)=1/\theta_2(\alpha)$;

\item [(ii)] $\theta_2\left(u_1+\alpha\right)=\theta_1\left(u_1\right)/\beta$;

\item[(iii)] $\kappa_2\left(u_1+\alpha\right)=\kappa_1\left(u_1\right)$.
\end{itemize}
\end{proposition}

\begin{proof}
Let us consider 
$$
\Psi\left(u_1,s_1\right)=\left(\Psi^1\left(u_1,s_1\right), \Psi^2\left(u_1,s_1\right)\right)	
$$
an isometry between $\left(M^2_1,g_1\right)$ and $\left(M^2_2,g_2\right)$, i.e., $\Psi^\ast g_2=g_1$. So, the following relations hold
\begin{equation}\label{eq1-isometry}
\left(\frac{\partial \Psi^1}{\partial u_1}\right)^2\left(u_1,s_1\right)+\left(\frac{\partial \Psi^2}{\partial u_1}\right)^2\left(u_1,s_1\right)\theta^2_2\left(\Psi^1\left(u_1,s_1\right)\right)=1,
\end{equation}
\begin{equation}\label{eq2-isometry}
\frac{\partial \Psi^1}{\partial u_1}\left(u_1,s_1\right) \frac{\partial \Psi^1}{\partial s_1}\left(u_1,s_1\right)+\frac{\partial \Psi^2}{\partial u_1}\left(u_1,s_1\right) \frac{\partial \Psi^2}{\partial s_1}\left(u_1,s_1\right)\theta^2_2\left(\Psi^1\left(u_1,s_1\right)\right)=0
\end{equation}
and
\begin{equation}\label{eq3-isometry}
\left(\frac{\partial \Psi^1}{\partial s_1}\right)^2\left(u_1,s_1\right)+\left(\frac{\partial \Psi^2}{\partial s_1}\right)^2\left(u_1,s_1\right)\theta^2_2\left(\Psi^1\left(u_1,s_1\right)\right)=\theta_1^2\left(u_1\right).
\end{equation}
Moreover, we also know that
\begin{equation}\label{eq4-isometry}
K_2\left(\Psi^1\left( u_1,s_1 \right)\right)=K_1\left(u_1\right).
\end{equation}
If we take the derivative of \eqref{eq4-isometry} with respect to $s_1$, since $K_2'\neq 0$, we easily obtain that $\Psi^1=\Psi^1\left(u_1\right)$. Further, knowing that the function $\Psi^1$ depends only on $u_1$, from \eqref{eq2-isometry}, we get
\begin{equation}\label{isom1}
\frac{\partial \Psi^2}{\partial u_1}\left(u_1,s_1\right) \frac{\partial \Psi^2}{\partial s_1}\left(u_1,s_1\right)=0,
\end{equation}
and, from \eqref{eq3-isometry}, we have
\begin{equation}\label{isom2}
\left(\frac{\partial \Psi^2}{\partial s_1}\right)^2\left(u_1,s_1\right)\theta^2_2\left(\Psi^1\left(u_1\right)\right)=\theta_1^2\left(u_1\right).
\end{equation}
Now, as $\theta_1>0$, from \eqref{isom1} and \eqref{isom2}, one obtains that $\Psi^2=\Psi^2\left(s_1\right)$, and then, from \eqref{eq1-isometry}, we have
$$
\Psi^1\left(u_1\right)=\pm u_1+\alpha, \qquad \alpha\in\mathbb{R}.
$$
Furthermore, taking the derivative of \eqref{eq4-isometry} with respect to $u_1$, as $K_1'$ and $K_2'$ are positive, it follows that $\left(\Psi^{1}\right)'>0$, so
$$
\Psi^1\left(u_1\right)=u_1+\alpha, \qquad \alpha\in\mathbb{R}.
$$
From \eqref{eq3-isometry}, we obtain
$$
\left(\left(\Psi^{2}\right)'\left(s_1\right)\right)^2=\frac{\theta_1^2\left(u_1\right)}{\theta_2^2\left(\Psi^1\left(u_1\right)\right)}.
$$
By taking the derivative of the above expression with respect to $s_1$, we get that  $\left(\Psi^{2}\right)'$ is constant and so
$$
\Psi^2\left(s_1\right)=\beta s_1+\gamma, \qquad \gamma\in\mathbb{R},
$$
where $\beta$ is a positive constant given by
\begin{equation}\label{eq-beta}
\beta=\frac{\theta_1\left(u_1\right)}{\theta_2\left(u_1+\alpha\right)}.
\end{equation}
In particular, using the assumptions that $\theta_1(0)=\theta_2(0)=1$, we get
$$
\beta=\theta_1(-\alpha)=\frac{1}{\theta_2(\alpha)}.
$$
Thus, $(i)$ and $(ii)$ are proved.

Finally, from \eqref{eq:kappa-theta'-theta} and \eqref{eq-beta},
we can conclude that
$$
\kappa_2\left(u_1+\alpha\right)=\kappa_1\left(u_1\right).
$$

\end{proof}

Conversely, by a standard argument, we have 

\begin{proposition}
Let $\left(u_0;\kappa_0, \kappa_0', \kappa_0''\right)$ and $\left(v_0;\kappa_0, \kappa_0', \kappa_0''\right)$ be two sets of initial conditions for \eqref{eq:kappa2} such that \eqref {initialConditions1} holds. Denote by $\kappa_1=\kappa_1\left(u_1\right)$ and $\kappa_2=\kappa_2\left(u_2\right)$ the corresponding solutions. Let $g_1=g_1\left(u_1,s_1\right)$ and $g_2=g_2\left(u_2,s_2\right)$ be the associated metrics, respectively. Then, 
$$
\Psi\left(u_1,s_1\right)=\left(u_1-\left(u_0-v_0\right),s_1\right)
$$
is an isometry, that is $\Psi^\ast g_2=g_1$.
\end{proposition}

From the above two results we can conclude that

\begin{corollary}
From a local point of view, we have a $3$-parameter family of non-isometric abstract surfaces $\left(M^2,g\right)$ that admit a $PNMC$ biconservative immersion in $\mathbb{S}^4$, indexed by $\kappa_0$, $\kappa_0'$ and $\kappa_0''$ that obey the initial conditions \eqref{intialConditions2}.
\end{corollary}

From a \textit{global} point of view (again, in our context of $PNMC$ biconservative surfaces in $\mathbb{S}^4$), we can perform a new change of coordinates such that we can see that the abstract surfaces $\left(M^2,g\right)$ that admit a $PNMC$ biconservative immersion in $\mathbb{S}^4$ form a family indexed by two parameters.

\begin{proposition}\label{other-parameters}
Let $\varphi:\left(M^2,g\right)\to \mathbb{S}^4$ be a $PNMC$ biconservative immersion. Then, around any point of $M^2$ there exists a positively oriented local chart $X^f=X^f(u,t)$ such that 
$$
E_1=X^f_u, \qquad E_2=f^{3/4}(u)X^f_t
$$
and
$$
g(u,t)=du^2+\frac{1}{f^{3/2}(u)}dt^2,
$$
where $f$ is the mean curvature function. Moreover, if we change the coordinates in an appropriate way, which allows a global point of view, the metric $g$ can be written as 
\begin{equation}\label{appropriate-coordinates}
g\left(f,t\right)=\frac{1}{2C^2f^{7/2}-\frac{16}{9}c^2f^5-16f^4-\frac{16}{9}f^2}df^2+\frac{1}{f^{3/2}}dt^2,
\end{equation}
where $C$ and $c$ are arbitrary non-zero real constants.
\end{proposition}

\begin{proof}
Let $\left\{E_1,E_2\right\}$ be the positively oriented global vector field given in Theorem \ref{thm:fundamentalProperties} and let $h$ be a smooth positive function on $M^2$, such that $E_2h=0$. From Theorem \ref{thm:fundamentalProperties} we have that $\left[E_1,hE_2\right]=0$ if and only if 
$$
h=A\frac{1}{f^{3/4}}, 
$$
where $A$ is an arbitrary positive real constant.

Therefore, there exists a positively oriented local chart $X^f=X^f\left(\tilde{u}, t\right)$ such that
$$
E_1=X^f_{\tilde{u}} \qquad \text{and} \qquad E_2=\frac{1}{h}X^f_t.
$$
It is not difficult to check that up to a translation, we have $\tilde{u}=u$, where $u$ is the local coordinate from the chart $X^f=X^f(u,v)$ in Theorem \ref{thm:fundamentalProperties}. Thus, 
$$
g(u,t)=du^2+A^2\frac{1}{f^{3/2}(u)}dt^2.
$$
By an appropriate change of the coordinate $t$, we may assume that the constant $A$ is equal to $1$, and so, keeping the same notation for the second coordinate, we can rewrite the metric $g$ as
$$
g(u,t)=du^2+\frac{1}{f^{3/2}(u)}dt^2.
$$
Next, if we perform a new change of coordinates $(u,t)\to (f=f(u), t)$, taking into account that $f$ satisfies \eqref{first-integral-f}, we infer that the metric $g$ takes the following form
$$
g\left(f,t\right)=\frac{1}{2C^2f^{7/2}-\frac{16}{9}c^2f^5-16f^4-\frac{16}{9}f^2}df^2+\frac{1}{f^{3/2}}dt^2,
$$
where $C$ and $c$ are arbitrary non-zero real constants.
\end{proof}

In the end of this section we note that the tangent part of the biharmonic equation has proved to be not so restrictive. But, if we want to work with both parts of the biharmonic equation, the situation is very rigid. Indeed, we have 
 
\begin{proposition}
A $PNMC$ biconservative surface in $\mathbb{S}^4$ cannot be biharmonic. 
\end{proposition}

\begin{proof}
Assume that there exists a biharmonic surface $M^2$ in $\mathbb{S}^4$ which is also $PNMC$.  From the first equation of \eqref{biharmonic-system}, by some straightforward computations, we get
$$
\left(\Delta f +10f^3-2f\right)E_3-4cf^{7/2}E_4=0.
$$
Since $c$ is a non-zero real constant, it follows that $f=0$ and this is a contradiction. 
\end{proof}

\section{An extrinsic approach}\label{sec-extrinsic}
Let $\varphi:\left(M^2,g\right)\to\mathbb{S}^4$ be a $PNMC$ biconservative immersion. Consider $i:\mathbb{S}^4\to\mathbb{E}^5$ be the canonical inclusion and denote 
$$
\Phi=i\circ \varphi:M^2\to\mathbb{E}^5.
$$
Recall that $NM^2$ represents the normal bundle of $\varphi$ and, in order to avoid any confusion, we denote by $N_\Phi M^2$ the normal bundle of the immersion $\Phi$. Clearly, the two normal bundles are related by
$$
N_\Phi M^2=i_\ast\left(NM^2\right)\oplus \Span \left\{\Phi\right\}
$$
and we have 
$$
B_\Phi(X,Y)=i_\ast(B(X,Y))-g(X,Y)\Phi, \qquad X,Y\in C\left(TM^2\right),
$$
where $B_\Phi$ denotes the second fundamental form of $\Phi$.

For any $\xi\in C\left(NM^2\right)$, we have $A_{i_\ast(\xi)}^\Phi=A_\xi$ and $A^\Phi_\Phi=-Id$.
Recall that $\left\{E_1,E_2\right\}$ is a positively oriented global orthonormal frame field in the tangent bundle $TM^2$ and $\left\{E_3,E_4\right\}$ is a positively oriented global orthonormal frame field in the normal bundle $NM^2$; thus, $\left\{E_3,E_4,\Phi\right\}$ is a positively oriented orthonormal frame field in the normal bundle $N_\Phi M^2$, where
$$
E_1=\frac{\grad f}{\left|\grad f\right|} \qquad \text{and} \qquad E_3=\frac{H}{f}.
$$
We recall that we denoted by $\tilde{\nabla}$ and $\hat{\nabla}$ the Levi-Civita connections of $\mathbb{S}^4$ and $\mathbb{E}^5$, respectively. Clearly,
$$
\hat{\nabla}_{\tilde{X}}\tilde{Y}=\tilde{\nabla}_{\tilde{X}}\tilde{Y}-\langle\tilde{X},\tilde{Y}\rangle \Phi, \qquad \tilde{X},\tilde{Y}\in C\left(T\mathbb{S}^4\right).
$$
As usual, when we work with isometric immersions, we will identify $M^2$ with its image, $X\in C\left(TM^2\right)$ with $\varphi_\ast (X)$ or with $\Phi_\ast (X)$, and the Riemannian metrics on $M^2$, $\mathbb{S}^4$, $\mathbb{E}^5$ will be simply denoted by $\langle,\rangle$. 

Our aim is to find extrinsic properties of $M^2$ and, finally, to infer a parametrization of $M^2$.

Now, using Theorem \ref{thm:fundamentalProperties}, by standard computations, we obtain 
\begin{equation}\label{LC-tilde,hat}
	\left\{
	\begin{array}{l}
		\hat{\nabla}_{E_1}E_1=\tilde{\nabla}_{E_1}E_1-\Phi=-fE_3+cf^{3/2}E_4-\Phi\\
		\hat{\nabla}_{E_2}E_1=\tilde{\nabla}_{E_2}E_1=-\frac{3}{4}\frac{E_1f}{f}E_2\\
		\hat{\nabla}_{E_1}E_2=\tilde{\nabla}_{E_1}E_2=0\\
		\hat{\nabla}_{E_2}E_2=\tilde{\nabla}_{E_2}E_2-\Phi=\frac{3}{4}\frac{E_1 f}{f}E_1+3fE_3-cf^{3/2}E_4-\Phi\\
		\hat{\nabla}_{E_1}E_3=\tilde{\nabla}_{E_1}E_3=fE_1\\
		\hat{\nabla}_{E_2}E_3=\tilde{\nabla}_{E_2}{E_3}=-3fE_2\\
		\hat{\nabla}_{E_1}E_4=\tilde{\nabla}_{E_1}E_4=-cf^{3/2}E_1\\
		\hat{\nabla}_{E_2}E_4=\tilde{\nabla}_{E_2}E_4=cf^{3/2}E_2\\
		\hat{\nabla}_{E_1}\Phi=E_1\\
		\hat{\nabla}_{E_2}\Phi=E_2
	\end{array}
	\right..
\end{equation} 

First, we will study the geometric properties of the integral curves of $E_1$, viewed as curves in $\mathbb{S}^4$. We denote such a curve by $\tilde{\gamma}=\tilde{\gamma}(u)$. When we consider this curve as a curve lying in $\mathbb{E}^5$, we set $\hat{\gamma}=i\circ\tilde{\gamma}$.

For a homogeneous notation of the vector fields that form the Frenet frame field associated to $\tilde{\gamma}$, we denote by $\tilde{V}_1$ the restriction of $E_1$ along $\gamma$. By straightforward computations we have
\begin{equation*}
\tilde{\nabla}_{\tilde{V}_1}\tilde{V}_1 = \tilde{\kappa}_1 \tilde{V}_2,
\end{equation*}
where we define
$$
\tilde{\kappa}_1=\left|\tilde{\nabla}_{\tilde{V}_1}\tilde{V}_1\right|=f\sqrt{1+c^2f}
$$
and
$$
\tilde{V}_2=\frac{1}{\tilde{\kappa}_1}\tilde{\nabla}_{\tilde{V}_1}\tilde{V}_1=\frac{1}{\sqrt{1+c^2f}}\left(-E_3+c\sqrt{f}E_4\right).
$$
Further,
\begin{equation*}
\tilde{\nabla}_{\tilde{V}_1}\tilde{V}_2 = -\tilde{\kappa}_1\tilde{V}_1+\tilde{\kappa}_2\tilde{V}_3,
\end{equation*}
where
$$
\tilde{\kappa}_2=\left|\tilde{\nabla}_{\tilde{V}_1}\tilde{V}_2+\tilde{\kappa}_1\tilde{V}_1\right|=\frac{|c|f'}{2\sqrt{f}\left(1+c^2f\right)}
$$
and
$$
\tilde{V}_3=\frac{1}{\tilde{\kappa}_2}\left(\tilde{\nabla}_{\tilde{V}_1}\tilde{V}_2+\tilde{\kappa}_1\tilde{V}_1\right)=\frac{|c|\sqrt{f}}{\sqrt{1+c^2f}}E_3+\frac{1}{\sqrt{1+c^2f}}E_4.
$$
Moreover, using \eqref{first-integral-f}, we obtain
$$
\tilde{\kappa}_2=\frac{|c|}{2\sqrt{f}\left(1+c^2f\right)}\sqrt{2C^2f^{7/2}-\frac{16}{9}f^2-16f^4-\frac{16}{9}c^2f^5}.
$$
Finally, 
\begin{equation*}
\tilde{\nabla}_{\tilde{V}_1}\tilde{V}_3 =-\tilde{\kappa}_2\tilde{V}_2= -\tilde{\kappa}_2\tilde{V}_2+\tilde{\kappa}_3\tilde{V}_4,
\end{equation*}
where  $\tilde{\kappa}_3=0$ and $\tilde{V}_4$ is the unit vector field along $\gamma$ such that $\left\{\tilde{V}_1,\tilde{V}_2,\tilde{V}_3,\tilde{V}_4\right\}$ is a positively oriented orthonormal frame field along $\gamma$.

As $\tilde{\kappa}_3=0$, it follows that $\tilde{\gamma}$ lies in a totally geodesic hypersphere $\mathbb{S}^3$ of $\mathbb{S}^4$, $\mathbb{S}^3=\mathbb{S}^4\cap \Pi$, where $\Pi$ is an hyperplane of $\mathbb{E}^5$ which contains the origin. More precisely, since 
$$
\hat{\nabla}_{\tilde{V}_1}E_2=0,
$$
i.e., $E_2$ is constant along $\hat{\gamma}$, and since $E_2$ is orthogonal to $\hat{\gamma}$ along $\gamma$, it follows that $E_2$ is normal to the hyperplane $\Pi$.

In conclusion, we have

\begin{proposition}\label{integral-curves-of-E1}
Let $\varphi:\left(M^2,g\right)\to\mathbb{S}^4$ be a $PNMC$ biconservative immersion and consider $\tilde{\gamma}=\tilde{\gamma}(u)$ an integral curve of $E_1$, viewed as a curve in $\mathbb{S}^4$. Then, the following hold:
\begin{itemize}
\item [(i)]$E_2$ is constant along $\hat{\gamma}$, where $\hat{\gamma}=i\circ\tilde{\gamma}$;
\item [(ii)] $\tilde{\gamma}$ lies in a totally geodesic hypersphere $\mathbb{S}^3=\mathbb{S}^4\cap \Pi$, where the hyperplane $\Pi$ contains the origin and is orthogonal to $E_2$;
\item[(iii)] the curvature and the torsion of $\tilde{\gamma}$ are
\begin{equation}\label{curvature-tilde-gammma}
\mathrm{k}(u)=f(u)\sqrt{1+c^2f(u)}
\end{equation} 
and
\begin{equation}\label{torsion-tilde-gamma}
\tau(u)=\frac{|c|}{2\sqrt{f(u)}\left(1+c^2f(u)\right)}\sqrt{2C^2f^{7/2}(u)-\frac{16}{9}f^2(u)-16f^4(u)-\frac{16}{9}c^2f^5(u)},
\end{equation}
where $f$ is the mean curvature function of the immersion $\varphi$.
\end{itemize}
\end{proposition}

Second, we will study the geometric properties of the integral curves of $E_2$, viewed as curves in $\mathbb{E}^5$.

Using the expression of $\hat{\nabla}_{E_2}E_2$ from \eqref{LC-tilde,hat} and \eqref{first-integral-f}, it is easy to see that
$$
\left|\hat{\nabla}_{E_2}E_2\right|^2=\frac{9C^2}{8}f^{3/2}.
$$
Denote by
\begin{equation}\label{curvature-integral-curve-E2}
\hat{\kappa}=\left|\hat{\nabla}_{E_2}E_2\right|=\frac{3|C|}{2\sqrt{2}}f^{3/4}
\end{equation}
and by 
\begin{equation}\label{xi}
\xi=\frac{1}{\hat{\kappa}}\hat{\nabla}_{E_2}E_2.
\end{equation}
Clearly, $E_2\hat{\kappa}=0$ and, using again \eqref{first-integral-f} and \eqref{LC-tilde,hat}, by some straightforward computations, we get
$$
\hat{\nabla}_{E_2}\xi=-\hat{\kappa}E_2, \qquad \frac{E_1\hat{\kappa}}{\hat{\kappa}}=\frac{3}{4}\frac{E_1f}{f}, \qquad \hat{\nabla}_{E_1}\xi=0.
$$
In conclusion, we have
\begin{proposition}\label{integral-curves-of-E2}
Let $\varphi:\left(M^2,g\right)\to\mathbb{S}^4$ be a $PNMC$ biconservative immersion and consider $\hat{\kappa}$ and $\xi$ given in \eqref{curvature-integral-curve-E2} and \eqref{xi}. Then, the following hold:
\begin{itemize}
\item [(i)] the integral curves of $E_2$ are circles in $\mathbb{E}^5$ with constant curvature $\hat{\kappa}$, i.e., $E_2\hat{\kappa}=0$ and
$$
\hat{\nabla}_{E_2}E_2=\hat{\kappa}\xi, \qquad \hat{\nabla}_{E_2}\xi=-\hat{\kappa}E_2;
$$
\item [(ii)] 
$$
\frac{E_1\hat{\kappa}}{\hat{\kappa}}=\frac{3}{4}\frac{E_1f}{f};
$$
\item [(iii)]
$$ 
\hat{\nabla}_{E_1}\xi=0.
$$
\end{itemize}
\end{proposition}
Now, we are ready to find the local parametrization of $M^2$ in $\mathbb{E}^5$. This parametrization will rely on a solution $f$ of a second order $ODE$ and on a certain curve in $\mathbb{S}^3$, uniquely determined by $f$ and the condition that its position vector has to make a specific angle with a constant direction.

For the beginning, let us consider an integral curve $\nu$ of $E_2$ parametrized by arc-length. As previously, when we will view this curve in $\mathbb{E}^5$ it will be denoted by $\hat{\nu}$. From Proposition \ref{integral-curves-of-E2}, we know that $\hat{\nu}$ is a circle in $\mathbb{E}^5$ with curvature $\hat{\kappa}$, thus it can be parametrized by
\begin{equation}\label{eq-circle}
\hat{\nu}(v)=C_0+\cos\left(\hat{\kappa}v\right)C_1+\sin\left(\hat{\kappa}v\right)C_2,
\end{equation}
where $C_0$, $C_1$, $C_2\in\mathbb{E}^5$ with $\left|C_1\right|=\left|C_2\right|=1/\hat{\kappa}$ and $\langle C_1,C_2\rangle=0$.

Then, let $p_0\in M^2$ be an arbitrarily fixed point of $M^2$ and $\hat{\gamma}=\hat{\gamma}(u)$ be an integral curve of $E_1$ with $\hat{\gamma}(0)=p_0$. Consider $\left\{\phi_v\right\}_{v\in\mathbb{R}}$ the flow of $E_2$ near the point $p_0$. Then, for any $u\in(-\varepsilon,\varepsilon)$ and for any $v\in\mathbb{R}$, the parametrization $\Phi=\Phi(u,v)$ of $M^2$ is given by
\begin{equation*}\label{Phi-1}
\Phi(u,v)=\phi_{\hat{\gamma}(u)}(v)=C_0(u)+\cos\left(\hat{\kappa}(u)v\right)C_1(u)+\sin\left(\hat{\kappa}(u)v\right)C_2(u),
\end{equation*}
where the vectorial functions $C_0(u)$, $C_1(u)$, $C_2(u)$, which are uniquely determined by the surface, satisfy
\begin{equation}\label{conditions-1}
\hat{\gamma}(u)=C_0(u)+C_1(u), \qquad \left|C_1 (u)\right|=\left|C_2(u)\right|=\frac{1}{\hat{\kappa}(u)}, \qquad  \langle C_1(u),C_2(u)\rangle=0.
\end{equation}
In order to get a simpler expression of $\Phi$ we consider the following change of coordinates $(u,v)\to \left(u, t=\hat{\kappa}(u)v\right)$. With respect to these new local coordinates, the parametrization $\Phi$ can be expressed as
\begin{equation*}\label{Phi-2}
\Phi(u,t)=C_0(u)+\frac{1}{\hat{\kappa}(u)}\left(\cos (t)  c_1(u)+\sin (t)c_2(u)\right),
\end{equation*}
where we define
\begin{equation}\label{link-ci-Ci}
c_i(u)=\hat{\kappa}(u)C_i(u), \qquad i=1,2.
\end{equation}
Using the above relation for $i=1$ and \eqref{conditions-1}, it is clear that 
$$
C_0(u)=\hat{\gamma}(u)-\frac{1}{\hat{\kappa}(u)}c_1(u).
$$
So, 
\begin{equation}\label{Phi-3}
\Phi(u,t)=\hat{\gamma}(u)+\frac{1}{\hat{\kappa}(u)}\left(\left(\cos (t)-1\right)c_1(u)+\sin (t) c_2(u)\right).
\end{equation}
Further, as $\hat{\nu}$ is an integral curve of $E_2$, from \eqref{eq-circle}, \eqref{link-ci-Ci}, it follows that 
$$
c_2(u)=E_2(u,0)
$$
and 
$$
c_1(u)=-\xi(u,0),
$$
where $E_2(u,0)=E_2\left(\nu(0)\right)$ and $\xi$ is given in \eqref{xi} with
$$
\xi(u,0)=\frac{1}{\hat{\kappa}(u)}\hat{\nabla}_{E_2(u,0)}E_2.
$$
Moreover, we will prove that, $c_1(u)$ and $c_2(u)$ are, in fact, constant vectors. Indeed, taking into account the third item from Proposition \ref{integral-curves-of-E2}, we get
$$
\hat{\nabla}_{\hat{\gamma}'}c_1=0,
$$
so $c_1$ is a constant vector in $\mathbb{E}^5$, and since $\hat{\nabla}_{E_1}E_2=0$, it follows that 
$$
\hat{\nabla}_{\hat{\gamma}'}c_2=0,
$$
so $c_2$ is a constant vector in $\mathbb{E}^5$. Therefore, \eqref{Phi-3} turns into

\begin{equation}\label{Phi-4}
\Phi(u,t)=\hat{\gamma}(u)+\frac{1}{\hat{\kappa}(u)}\left(\left(\cos(t)-1\right)c_1+\sin(t)c_2\right),	
\end{equation}
where $c_1$ and $c_2$ are two constant orthonormal vectors in $\mathbb{E}^5$ and $\hat{\gamma}$ is an integral curve of $E_1$ with the properties given in Proposition \ref{integral-curves-of-E1}. In particular, $\hat{\gamma}$ lies in a totally geodesic hypersphere $\mathbb{S}^3=\mathbb{S}^4\cap \Pi$, where the hyperplane $\Pi$ contains the origin and is orthogonal to $c_2$.

Finally, from $\left|\Phi(u,t)\right|=1$, we get
$$
\langle\hat{\gamma}(u),c_1\rangle=\frac{1}{\hat{\kappa}(u)}.
$$
\begin{remark}
We note that, up to a multiplicative constant, the parameter $t$ in formula \eqref{Phi-4} coincides with the parameter $t$ in Proposition \ref{other-parameters}.
\end{remark}
In conclusion we can state

\begin{theorem}\label{theorem-parametrization}
Let $\varphi:\left(M^2,g\right)\to\mathbb{S}^4$ be a $PNMC$ biconservative immersion and denote $\Phi=i\circ\varphi:M^2\to\mathbb{E}^5$, where $i:\mathbb{S}^4\to\mathbb{E}^5$ is the canonical inclusion. We identify $M^2$ with its image, and then, $M^2$ can be locally parametrized by
$$
\Phi(u,t)=\hat{\gamma}(u)+\frac{1}{\hat{\kappa}(u)}\left(\left(\cos(t)-1\right)c_1+\sin(t)c_2\right),	
$$ 
where
\begin{itemize}
\item[(i)] 
$$
\hat{\kappa}(u)=\frac{3|C|}{2\sqrt{2}}f^{3/4}(u),
$$
where $f=f(u)$ is a positive solution of the second order ODE \eqref{second-order-chart-f}, with $f'>0$, and whose first integral is \eqref{first-integral-f}. The non-zero constant $C$ is given in \eqref{first-integral-f}; 
\item[(ii)] $c_1$ and $c_2$ are constant orthonormal vectors in $\mathbb{E}^5$;
\item[(iii)] $\hat{\gamma}=\hat{\gamma}(u)$ is a curve in $\mathbb{E}^5$ such that $\hat{\gamma}=i\circ \tilde{\gamma}$, where $\tilde{\gamma}$ is a curve parametrized by arc-length which lies in a totally geodesic hypersphere $\mathbb{S}^3=\mathbb{S}^4\cap \Pi$;  the hyperplane $\Pi$ contains the origin and is orthogonal to $c_2$. Moreover, the curvature and torsion of $\tilde{\gamma}$, as a curve in $\mathbb{S}^3$, are given by \eqref{curvature-tilde-gammma} and \eqref{torsion-tilde-gamma}, respectively, and the curve $\hat{\gamma}$ must satisfy
\begin{equation}\label{eq3}
\langle\hat{\gamma}(u),c_1\rangle=\frac{1}{\hat{\kappa}(u)}.
\end{equation}
\end{itemize}
\end{theorem}

\begin{remark}
Theorems \ref{thm:intrinsicCharacterization}, \ref{thm-existence} and Proposition \ref{other-parameters} prove the existence of the $PNMC$ biconservative surfaces in $\mathbb{S}^4$, and therefore the existence of the curves $\hat{\gamma}$ that satisfy the hypotheses of the above theorem is ensured.
\end{remark}

\begin{remark}
If we consider a surface $M^2$ in $\mathbb{S}^4$ parametrized as in Theorem \ref{theorem-parametrization} then, by a straightforward computation, we can prove that it is a $PNMC$ biconservative surface with $f>0$, $\grad f\neq 0$ at any point and $M^2$ is completely contained in $\mathbb{S}^4$.
\end{remark}

Now, we can ask whether all three conditions which determine the curve $\tilde{\gamma}$, i.e.,  \eqref{curvature-tilde-gammma}, \eqref{torsion-tilde-gamma} and \eqref{eq3} are independent. The answer is negative and, more precisely, we have:

\begin{proposition}
Let $\tilde{\gamma}$  be a curve which satisfies \eqref{curvature-tilde-gammma} and \eqref{eq3}. Then, its torsion is given by \eqref{torsion-tilde-gamma}.
\end{proposition}

\begin{proof}
We can assume that $c_1=e_1$ and $c_2=e_2$, where $\left\{e_1,e_2,e_3,e_4,e_5\right\}$ is the canonical basis of $\mathbb{E}^5$.

Let $\hat{\gamma}=i\circ\tilde{\gamma}$ be a curve parametrized by arc-length. From \eqref{eq3}, $\hat{\gamma}$ can be written as
$$
\hat{\gamma}=\left(\gamma^1,0,\gamma^2,\gamma^3,\gamma^4\right)\equiv \left(\gamma^1,\gamma^2,\gamma^3,\gamma^4\right),
$$
where 
$$
\gamma^1=\frac{1}{\hat{\kappa}(u)}.
$$
With the above notations, we write down the Frenet formulas
\begin{equation*}
\left\{
\begin{array}{l}
\tilde{\nabla}_{\tilde{V}_1}\tilde{V}_1=\mathrm{k}\tilde{V}_2 \\
\tilde{\nabla}_{\tilde{V}_1}\tilde{V}_2=-\mathrm{k}\tilde{V}_1+\tau\tilde{V}_3 \\
\tilde{\nabla}_{\tilde{V}_1}\tilde{V}_3=-\tau\tilde{V}_2
\end{array}
\right.,
\end{equation*}
where 
\begin{equation*}
\tilde{V}_1=\tilde{\gamma}', \qquad \mathrm{k}=\left|\tilde{\nabla}_{\tilde{V}_1}\tilde{V}_1\right|, \qquad \tau=\left|\tilde{\nabla}_{\tilde{V}_1}\tilde{V}_2+\mathrm{k}\tilde{V}_1\right|.
\end{equation*}
Using the second fundamental form of $\mathbb{S}^4$ in $\mathbb{E}^5$, or of $\mathbb{S}^3$ in $\mathbb{E}^4$, we obtain
\begin{equation*}
\left\{
\begin{array}{l}
\tilde{\nabla}_{\tilde{V}_1}\tilde{V}_1=\hat{\gamma}''+\hat{\gamma}\\
\tilde{\nabla}_{\tilde{V}_1}\tilde{V}_2=\hat{V}'_2 \\
\tilde{\nabla}_{\tilde{V}_1}\tilde{V}_3=\hat{V}'_3
\end{array}
\right..
\end{equation*} 
Therefore, we get
\begin{equation}\label{hatV2}
\hat{V}_2=\frac{1}{\mathrm{k}}\tilde{\nabla}_{\tilde{V}_1}\tilde{V}_1=\frac{1}{\mathrm{k}}\left(\hat{\gamma}''+\hat{\gamma}\right)
\end{equation}
and
\begin{equation*}
\left\{
\begin{array}{l}
\hat{V}'_1=\mathrm{k}\hat{V}_2-\hat{\gamma}\\
\hat{V}'_2=-\mathrm{k}\hat{V}_1+\tau \hat{V}_3 \\
\hat{V}'_3=-\tau \hat{V}_2
\end{array}
\right..
\end{equation*} 
Now, we can write each vector $\hat{V}_i$ as
$$
\hat{V}_i= \left(V^1_i,0,V^2_i,V^3_i,V^4_i\right)\equiv \left(V^1_i,V^2_i,V^3_i,V^4_i\right), \qquad i=1,2,3.
$$
Looking at the first component of \eqref{hatV2} and using \eqref{second-order-chart-f}, we obtain 
$$
V_2^1= \frac{1}{\mathrm{k}}\left(\left(\gamma^1\right)''+\gamma^1\right)=\frac{2\sqrt{2}\left(3+c^2f\right)f^{1/4}}{3|C|\sqrt{1+c^2f}}>0.
$$
Since at any point of $\tilde{\gamma}$ the vectors $\left\{\hat{V}_1,\hat{V}_2,\hat{V}_3, \hat{\gamma}\right\}$ form an orthonormal basis of the tangent space to $\mathbb{E}^4$, we have
$$
\left(V^1_1\right)^2+\left(V^1_2\right)^2+\left(V^1_3\right)^2+\left(\gamma^1\right)^2=1.
$$
Thus, by some straightforward computations, using  \eqref{first-integral-f}, we get
$$
\left(V^1_3\right)^2=\frac{32c^2f^{3/2}}{9C^2\left(1+c^2f\right)}.
$$
Next, since $\left(V^1_3\right)'=-\tau V_2^1$ and $\tau V_2^1$ is positive, we get
$$
V^1_3=-\frac{4\sqrt{2}|c|f^{3/4}}{3|C|\sqrt{1+c^2f}}.
$$ 
and
$$
\tau=\frac{|c|f'}{2\sqrt{f}\left(1+c^2f\right)}.
$$
Finally, from \eqref{first-integral-f}, it follows that the torsion of $\tilde{\gamma}$ is given by \eqref{torsion-tilde-gamma}.
\end{proof}

\end{document}